\documentclass[11pt]{article}

\usepackage{amsthm,amssymb,amsmath}
\usepackage{color}
\usepackage[T1]{fontenc}
\usepackage[utf8]{inputenc}
\usepackage{enumitem}
\usepackage{url}
\usepackage{bbold}
\usepackage{etex}
\usepackage{memhfixc}
\usepackage{titlesec}
\titleformat{\subsection}
  {\normalfont\bf\itshape}{\thesubsection}{1em}{}
  \titleformat{\subsubsection}
  {\normalfont\itshape}{\thesubsubsection}{1em}{}
\usepackage[top=2cm,bottom=2cm,left=2cm,right=2cm]{geometry}

\numberwithin{equation}{section}

\usepackage{hyperref}
\hypersetup{
  colorlinks = true,
  urlcolor = blue, 
  linkcolor = blue,
  citecolor = blue}

\newtheorem{thm}{Theorem}[section]
\newtheorem{rem}[thm]{Remark}
\newtheorem{defi}[thm]{Definition}
\newtheorem{lem}[thm]{Lemma}

\newtheorem{pro}[thm]{Proposition}
\newtheorem{exm}[thm]{Example}

\newcommand{\prob}{\mathbb{P}}
\newcommand{\esp}{\mathbb{E}}

\newcommand{\N}{\mathbb N}

\newcommand{\R}{\mathbb R}

\newcommand{\e}{\varepsilon}

\newcommand{\myand}{\quad\mbox{and}\quad}

\newcommand{\supp}{\textrm{supp}}

\newcommand{\ms}{M}
\newcommand{\curv}{\mathrm{curv}}
\newcommand{\Ric}{\mathrm{Ric}}
\newcommand{\Hess}{\mathrm{Hess}}

\title{Convergence rates for empirical barycenters in metric spaces: curvature, convexity and extendable geodesics}
\author{A. Ahidar-Coutrix\footnote{Aix Marseille Univ., CNRS, Centrale Marseille, I2M, Marseille, France. Email:adil.ahidar@outlook.com},\, T. Le Gouic\footnote{Aix Marseille Univ., CNRS, Centrale Marseille, I2M, Marseille, France \& National Research University Higher School of Economics, Moscow, Russia. This work has been funded by the  Russian Academic Excellence Project '5-100'. Email:thibaut.le$\_$gouic@math.cnrs.fr} \,and Q. Paris\footnote{National Research University Higher School of Economics, Moscow, Russia. This work has been funded by the  Russian Academic Excellence Project '5-100'. Email:qparis@hse.ru}}
\date{\today}

\begin{document}
\maketitle

\begin{abstract}
This paper provides rates of convergence for empirical (generalised) barycenters on compact geodesic metric spaces under general conditions using empirical processes techniques.
Our main assumption is termed a variance inequality and provides a strong connection between usual assumptions in the field of empirical processes and central concepts of metric geometry.
We study the validity of variance inequalities in spaces of non-positive and non-negative Aleksandrov curvature.
In this last scenario, we show that variance inequalities hold provided geodesics, emanating from a barycenter, can be extended by a constant factor.
We also relate variance inequalities to strong geodesic convexity. While not restricted to this setting, our results are largely discussed in the context of the $2$-Wasserstein space.
\end{abstract}

\newpage
\section{Introduction}
Given a separable and complete metric space $(M,d)$, define $\mathcal P_2(M)$ as the set of Borel probability measures $P$ on $M$ such that 
\[
\int_M d(x,y)^2\,\textrm{d}P(y)<+\infty,
\]
for all $x\in M$.
A barycenter of $P\in \mathcal P_2(M)$, also called a Fr\'echet mean \cite{Fr48}, is any element $x^*\in M$ such that
\begin{equation}
\label{xstarintrobary}
x^{*}\in\underset{x\in M}{\arg\min}\,\int_M d(x,y)^2\,\textrm{d}P(y).
\end{equation}
When it exists, a barycenter stands as a natural analog of the mean of a (square integrable) probability measure on $\R^d$.
Alternative notions of mean value include local minimisers \cite{Kar14}, $p$-means \cite{yokota2017}, exponential barycenters \cite{emery1991barycentre} or convex means \cite{emery1991barycentre}.
Extending the notion of mean value to the case of probability measures on spaces $M$ with no Euclidean (or Hilbert) structure has a number of applications ranging from geometry \cite{sturm2003} and optimal transport \cite{villani2003,villani2008optimal,santambrogio2015,cp2018} to statistics and data science \cite{Pelletier2005,BLL15,Bigotandco2018,KSS19}, and the context of abstract metric spaces provides a unifying framework encompassing many non-standard settings.\\

Properties of barycenters, such as existence and uniqueness, happen to be closely related to geometric characteristics of the space $M$.
These properties are addressed in the context of Riemannian manifolds in \cite{Af11}.
Many interesting examples of metric spaces, however, cannot be described as smooth manifolds because of their singularities or infinite dimensional nature.
More general geometrical structures are geodesic metric spaces which include many more examples of interest (precise definitions and necessary background on metric geometry are reported in Appendix \ref{app:mg}).
The barycenter problem has been addressed in this general setting.
The scenario where $M$ has non-positive curvature (from here on, curvature bounds are understood in the sense of Aleksandrov) is considered in \cite{sturm2003}.
More generally, the case of metric spaces with upper bounded curvature is studied in \cite{yokota2016} and \cite{yokota2017}.
The context of spaces $M$ with lower bounded curvature is discussed in \cite{yokota2012rigidity} and \cite{ohta2012barycenters}.\\

Focus on the case of metric spaces with non-negative curvature may be motivated by the increasing interest for the theory of optimal transport and its applications.
Indeed, a space of central importance in this context is the Wasserstein space $M=\mathcal P_2(\R^d)$, equipped with the Wasserstein metric $W_2$, known to be geodesic and with non-negative curvature (see Section 7.3 in \cite{ambrosio2008gradient}).
In this framework, the barycenter problem was first studied by \cite{agueh2010barycenters} and has since gained considerable momentum.
Existence and uniqueness of barycenters in $\mathcal P_2(\R^d)$ has further been studied in \cite{Legouic2017}. \\

A number of objects of interest, including barycenters as a special case, may be described as minimisers of the form 
\begin{equation}
\label{xstarintro}
x^{*}\in\underset{x\in M}{\arg\min}\,\int_M F(x,y)\,\textrm{d}P(y),
\end{equation}
for some probability measure $P$ on metric space $M$ and some functional $F:M\times M\to \R$.
While we obviously recover the definition of barycenters whenever $F(x,y)=d(x,y)^2$, many functionals of interest are not of this specific form.
With a slight abuse of language, minimisers such as $x^*$ will be called generalised barycenters in the sequel.
A first example we have in mind, in the context where $M=\mathcal P_2(\R^d)$, is the case where functional $F$ is an $f$-divergence, i.e. 
\begin{equation}
\nonumber
F(\mu,\nu):=\left\{\begin{array}{ll}
\int f\left(\frac{{\rm d}\mu}{{\rm d}\nu}\right)\,{\rm d}\nu &\mbox{ if }\mu\ll \nu,\\
+\infty&\mbox{ otherwise,}\
\end{array}\right.
\end{equation}
for some convex function $f:\R_+\to \R$.
Known for their importance in statistics \cite{LeCam86,Tsyb09}, and information theory \cite{Vaj89}, $f$-divergences have become a crucial tool in a number of other fields such as geometry and optimal transport \cite{Sturm06I,Sturm06II,LottVillani2009} or machine learning \cite{Goodfellow2014}.
Other examples arise when the squared distance $d(x,y)^2$ in \eqref{xstarintrobary} is replaced by a regularised version $F(x,y)$ aiming at enforcing computationally friendly properties, such as convexity, while providing at the same time a sound approximation of $d(x,y)^2$.
A significant example in this spirit is the case where functional $F$ is the entropy-regularised Wasserstein distance (also known as the Sinkhorn divergence) largely used as a proxy for $W_2$ in applications \cite{cuturi2013,cp2018,Altschuler2017,Dvurechensky2018}.\\

In the paper, our main concern is to provide rates of convergence for empirical generalised barycenters, defined as follows.
Given a collection $Y_1,\dots,Y_n$ of independent and $M$-valued random variables with same distribution $P$, we call empirical generalised barycenter any
\begin{equation}
\label{xnintro}
x_n\in\underset{x\in\ms}{\arg\min}\,\frac{1}{n}\sum_{i=1}^{n}F(x,Y_i).
\end{equation}
Any such $x_n$ provides a natural empirical counterpart of a generalised barycenter $x^*$ defined in \eqref{xstarintro}.
The statistical properties of $x_n$ have been studied in a few specific scenarios.
In the case where $F(x,y)=d(x,y)^2$ and $M$ is a Riemannian manifold, significant contributions, establishing in particular consistency and limit distribution under general conditions, are \cite{Bhattacharya2003,Bhattacharya2005} and \cite{Kendall2011}.
Asymptotic properties of empirical barycenters in the Wasserstein space are studied in \cite{Legouic2017}.
We are only aware of a few contributions providing finite sample bounds on the statistical performance of $x_n$.
Paper \cite{Bigotandco2018} provides upper and lower bounds on convergence rates for empirical barycenters in the context of the Wasserstein space over the real line.
Independently of the present contribution, \cite{Sh18} studies a similar problem and provides results complementary to ours.\\

In addition to more transparent conditions, our results are based on the fundamental assumption that there exists constants $K>0$ and $\beta\in(0,1]$ such that, for all $x\in M$,
\begin{equation}
\label{ivintro}
d(x,x^*)^2\le K\left(\int_M (F(x,y)-F(x^*,y))\,\mathrm{d}P(y)\right)^{\beta}.
\end{equation}
We show that condition \eqref{ivintro} provides a connection between usual assumptions in the field of empirical processes and geometric characteristics of the metric space $M$.
First, the reader familiar with the theory of empirical processes will identify in the proof of Theorems \ref{thm:upperbound1} and \ref{thm:upperbound2} that condition \eqref{ivintro} implies a Bernstein condition on the class of functions indexing our empirical process, that is an equivalence between their $L_2$ and $L_1$ norms.
Many authors have emphasised the role of this condition for obtaining fast rates of convergence of empirical minimisers.
Major contributions in that direction are for instance \cite{mammen1999,massart2000,blanchard2003,bartlett2005,bartlett2006,koltchinskii2006,bandm2006} and \cite{Men15}.
In particular, this assumption may be understood in our context as an analog of the Mammen-Tsybakov low-noise assumption \cite{mammen1999} used in binary classification.
Second, we show that condition \eqref{ivintro} carries a strong geometrical meaning.
In the context where $F(x,y)=d(x,y)^2$, \cite{sturm2003} established a tight connection between \eqref{ivintro}, with $K=\beta=1$, and the fact that $M$ has non-positive curvature.
When $F(x,y)=d(x,y)^2$, we show that \eqref{ivintro} actually holds with $K>0$ and $\beta=1$ in geodesic spaces of non-negative curvature under flexible conditions related to the possibility of extending geodesics emanating from a barycenter.
Finally, for a general functional $F$, we connect \eqref{ivintro} to its strong convexity properties.
Using terminology introduced in \cite{sturm2003} in a slightly more specific context, we will call by extention \eqref{ivintro} a variance inequality.\\

The paper is organised has follows.
Section \ref{sec:rates} provides convergence rates for generalised empirical barycenters under several assumptions of functional $F$ and two possible complexity assumptions on metric space $M$.
Section \ref{sec:vi} investigates in details the validity of the variance inequality \eqref{ivintro} in different scenarios.
In particular, we focus on studying \eqref{ivintro} under curvature bounds of the metric $M$ whenever $F(x,y)=d(x,y)^2$.
Additional examples where our results apply are discussed in Section \ref{sec:exmps}.
Proofs are postponed to Section \ref{sec:proofs}.
Finally, Appendix \ref{app:mg} presents an overview of basic concepts and results in metric geometry for convenience.

\section{Rates of convergence}
\label{sec:rates}
In this section, we provide convergence rates for generalised empirical barycenters.
Paragraph \ref{subsec:setup} defines our general setup and mentions our main assumptions on functional $F$.
Paragraphs \ref{subsec:doubling} and \ref{subsec:pme} present rates under different assumptions on the complexity of metric space $M$.
Subsection \ref{subsec:opt} discusses the optimality of our results.

\subsection{Setup}
\label{subsec:setup}
Let $(\ms,d)$ be a separable and complete metric space and $F:\ms\times\ms\to \R$ a measurable function.
Let $P$ be a Borel probability measure on $M$.
Suppose that, for all $x\in\ms$, the function $y\in\ms\mapsto F(x,y)$ is integrable with respect to $P$, and let
\begin{equation}
\label{eq:xstar}
x^*\in\underset{x\in\ms}{\arg\min}\,\int_{\ms} F(x,y)\,\textrm{d}P(y),
\end{equation}
which we suppose exists.
Given a collection $Y_1,\dots,Y_n$ of independent and $M$-valued random variables with same distribution $P$, we consider an empirical minimiser
\begin{equation}
\label{eq:xn}
x_n\in\underset{x\in\ms}{\arg\min}\,\frac{1}{n}\sum_{i=1}^{n}F(x,Y_i).
\end{equation}
The present section studies the statistical performance of $x_n$ under the following assumptions on $F$.

\begin{enumerate}[label=(A\arabic*),leftmargin=*]
\item\label{A1} There exists a constant $K_1>0$ such that, for all $x,y\in\ms$,
\[
|F(x,y)|\le K_1.
\]
\item\label{A2} There exist constants $K_2>0$ and $\alpha\in(0,1]$ such that, for all $x,x',y\in\ms$, 
\[
|F(x,y)-F(x',y)|\le K_2 d(x,x')^{\alpha}.
\]
\item\label{A3} (Variance inequality) There exist constants $K_3>0$ and $\beta\in(0,1]$ such that, for all $x\in\ms$, 
\[
d(x,x^*)^2\le K_3\left(\int_{\ms}(F(x,y)-F(x^*,y))\,{\rm d}P(y)\right)^{\beta}.
\]
\end{enumerate}

Assumptions \ref{A1} and \ref{A2} are transparent boundedness and regularity conditions.
For instance if $F(x,y)=d(x,y)^2$, these assumptions are satisfied whenever $\ms$ is bounded with $K_1=\mathrm{diam}(M)^2$, $K_2=2\mathrm{diam}(M)$ and $\alpha=1$, by the triangular inequality.
The meaning of condition \ref{A3} is less obvious at first sight.
A detailed discussion of \ref{A3} is postponed to section \ref{sec:vi}.
For now, we mention three straightforward implications of \ref{A3}.
First, note that imposing both \ref{A2} and \ref{A3} requires $M$ to be bounded.
Indeed, plugging \ref{A2} into \ref{A3} yields
\[\mathrm{diam}(M)\le 2(K^{\beta}_2K_3)^{\frac{1}{2-\alpha\beta}}.\]
More importantly, \ref{A3} implies that minimiser $x^*$ is unique.
Finally, condition \ref{A3} applied to minimiser $x_n$ reads 
\begin{equation}
    \label{eq:risk}
d(x_n,x^*)^2\le K_3\left(\int_M(F(x_n,y)-F(x^*,y))\,{\rm d}P(y)\right)^{\beta}.
\end{equation}
The left hand side of this inequality is the estimation performance of $x_n$.
The integral, under power $\beta$, may be called the learning performance of $x_n$, or its excess risk.
Having this comparison in mind, we will focus on controlling the learning performance of $x_n$ knowing that an upper bound on $d(x_n,x^*)^2$ may be readily deduced from our results. The remainder of the section therefore presents upper bounds for the right hand side of \eqref{eq:risk} under specific complexity assumptions on $\ms$.
For that purpose, we recall the definition of covering numbers.
For $A\subset\ms$ and $\e>0$, an $\e$-net for $A$ is a finite subset $\{x_1,\dots,x_N\}\subset\ms$ such that 
\[A\subset \bigcup_{j=1}^N B(x_j,\e),\]
where $B(x,\e):=\{u\in\ms:d(x,u)< \e\}$.
The $\e$-covering number $N(A,d,\e)\in (1,+\infty]$ is the smallest integer $N\ge1$ such that there exists an $\e$-net of size $N$ for $A$ in $\ms$.
The function $\e\mapsto\log N(A,d,\e)$ will be referred to as the metric entropy of $A$.

\subsection{Doubling condition}
\label{subsec:doubling}

Our first complexity assumption is the following.
\begin{enumerate}[label=(B1),leftmargin=*]
\item\label{B1} (Doubling condition) There exist constants $C,D>0$ such that, for all $0<\e\le r$,
\[N(B(x^*, r),d,\e)\le \left(\frac{Cr}{\e}\right)^D.\]
\end{enumerate}

Condition \ref{B1} essentially characterises $M$ as a $D$-dimensional space and implies the following result.
\begin{thm}
 \label{thm:upperbound1}
Assume that \ref{A1}, \ref{A2}, \ref{A3} and \ref{B1} hold.
Then, for all $n\ge 1$ and all $t>0$,
\[
kd(x^*,x_n)^{\frac{2}{\beta}}\le  \int_M(F(x_n,y)-F(x^*,y))\,{\rm d}P(y) \le A\cdot\max\left\{\left(\frac{D}{n}\right)^{\frac{1}{2-\alpha\beta}},\left(\frac{t}{n}\right)^{\frac{1}{2-\alpha\beta}}\right\},
\]
with probability at least $1-2e^{-t}$, where $k=K_3^{-\frac{1}{\beta}}$ and $A$ is an explicit constant independent of $n$.
\end{thm}

Note that bounds in expectation may be derived from this result, using classical arguments.
As described in section \ref{sec:vi}, \ref{A2} and \ref{A3} hold in several interesting cases for $\alpha=\beta=1$.
In this case, Theorem \ref{thm:upperbound1} exhibits an upper bound of order $D/n$.
A discussion on the optimality of this result is postponed to paragraph \ref{subsec:opt} below.
Next, we shortly comment condition \ref{B1}.
\begin{rem}
With a slight abuse of terminology, condition \ref{B1} is termed doubling condition.
In the literature, the doubling condition usually refers to the situation where inequality
\[
\sup_{x\in M}\sup_{\e>0}\log_2 N(B(x,2\e),d,\e)\le D
\]
holds for some $D\in(0,+\infty)$.
It may be seen that this inequality implies \ref{B1} with $C=2$.
Note however that \ref{B1} is slightly less restrictive as it requires only the control of the covering numbers of balls centered at $x^*$.
This fact is sometimes useful as described in example \ref{exm:lsfamilies} below.
\end{rem}
We now give two examples where assumption \ref{B1} holds.
\begin{exm}
\label{exm:finited}
 Suppose there exists a positive Borel measure $\mu$ on $(\ms,d)$ such that, for some $D>0$, we have
\begin{equation}
\label{exm:finite2}
\forall (x,r)\in M\times(0,+\infty),\quad \alpha_-r^{D}\le \mu(B(x,r))\le \alpha_+r^{D},
\end{equation}
for some constants $0<\alpha_-\le\alpha_+<+\infty$.
Then, for all $x\in\ms$ and all $0<\e\le r$, 
\begin{equation}
\label{exm:finite3}
\frac{\alpha_-}{\alpha_+}\left(\frac{r}{\e}\right)^D\le N(B(x,r),d,\e)\le \frac{\alpha_+}{\alpha_-}\left(\frac{3r}{\e}\right)^D,
\end{equation}
and thus \ref{B1} is satisfied.
The proof is given in Section \ref{sec:proofs}.
Measures $\mu$ satisfying condition \eqref{exm:finite2} are called $D$-regular or Ahlfors-David regular.
Many examples of such spaces are discussed in section 12 in \cite{GrLu00} or section 2.2 in \cite{AT04}.
Note that the present example includes the case where $M$ is a $D$-dimensional compact Riemannian manifold equipped with the volume measure $\mu$.
\end{exm} 

A direct and simple consequence of Example \ref{exm:finited} is that \ref{B1} holds in any $D$-dimensional vector space equipped with any norm since the Lebesgue measure satisfies \eqref{exm:finite2} with $\alpha_-=\alpha_+$.
While simple in essence, this observation allows to exhibit more general parametric families satisfying \ref{B1} as in the next example.

\begin{exm}[Location-scatter family]
\label{exm:lsfamilies}
Here, we detail an example of a subset $M$ of the Wasserstein space $(\mathcal P_2(\R^d),W_2)$ for which assumption \ref{B1} holds.
We say that $M\subset \mathcal P_2(\R^d)$ is a location-scatter family if the following two requirements hold:
\begin{itemize}
    \item[$(1)$] All elements of $M$ have a non-singular covariance matrix.
    \item[$(2)$] For every two measures $\mu_0,\mu_1\in M$, with expectations $m_0$ and $m_1$ and with covariance matrices $\Sigma_0$ and $\Sigma_1$ respectively, the map
\[
T_{01}:x\mapsto(m_1-m_0)+\Sigma_0^{-1/2}\left(\Sigma_0^{1/2}\Sigma_1\Sigma_0^{1/2}\right)^{1/2}\Sigma_0^{-1/2}x
\]
pushes forward $\mu_0$ to $\mu_1$, i.e. $\mu_0(T^{-1}_{01}(A))=\mu_1(A)$ for any Borel set $A\subset \R^d$, which we denote $(T_{01})_{\#}\mu_0=\mu_1$.
\end{itemize}
Such sets have been studied for instance in \cite{alvarez2016fixed}.
The map $T_{01}$ being the gradient of a convex function, the theory of optimal transport guarantees that the coupling $(\mathrm{id}, T_{01})_{\#}{\mu_0}$ is optimal, so that
\begin{align}
    W_2^2(\mu_0,\mu_1)&=\int \|x-T_{01}(x)\|^2{\rm d}\mu_0(x)\nonumber\\
    &=\|m_0-m_1\|^2+\mathrm{tr}\left(\Sigma_0 +\Sigma_1 - 2\left(\Sigma_0^{1/2}\Sigma_1\Sigma_0^{1/2}\right)^{1/2}\right),
    \label{eq:locscat0}
\end{align}
where $\mathrm{tr}(A)$ denotes the trace of matrix $A$ and $\|.\|$ refers to the standard euclidean norm.
Next, we show that such a family satisfies \ref{B1}.
Let $P$ be a probability measure on $M$ and denote $\mu^*\in M$ a barycenter of $P$ with mean $m_*$ and covariance matrix $\Sigma_*$.
For any two measures $\mu_0,\mu_1\in M$, set
\[
T_{*i}(x)=m_i-m_*+ \Sigma_*^{-1/2}\left(\Sigma_*^{1/2}\Sigma_i\Sigma_*^{1/2}\right)^{1/2}\Sigma_*^{-1/2}x,
\]
where $m_i$ and $\Sigma_i$ denote the mean and covariance matrix of $\mu_i$, $i=0,1$.
The pushforward $(T_{*0},T_{*1})_{\#}\mu^*$ is a (possibly suboptimal) coupling between $\mu_0$ and $\mu_1$.
Therefore, 
\begin{align}
W_2^2(\mu_0,\mu_1)&\le \int \|T_{*0}(x)-T_{*1}(x)\|^2 {\rm d}\mu^*(x)\label{eq:locscat}\\
&=\|m_0-m_1\|^2 + \|\Sigma_*^{-1/2}\left(\Sigma_*^{1/2}\Sigma_0\Sigma_*^{1/2}\right)^{1/2}-\Sigma_*^{-1/2}\left(\Sigma_*^{1/2}\Sigma_1\Sigma_*^{1/2}\right)^{1/2}\|_{F}^2,\nonumber
\end{align}
where $\|.\|_{F}$ stands for the Frobenius norm.
Note that $\|(m,A)\|^2_*= \|m\|^2+\|A\|_F^2$ defines a norm $\|.\|_*$ on the vector space $\R^d\times S_d$ where $S_d$ denotes the space of symmetric matrices of size $d\times d$.
Then, define the function $\phi$ that maps each $\mu_{m,\Sigma}$ in the location-scatter family, with mean $m$ and covariance $\Sigma$, to
\[
\phi(\mu_{m,\Sigma})=\left(m,\Sigma_*^{-1/2}\left(\Sigma_*^{1/2}\Sigma\Sigma_*^{1/2}\right)^{1/2}\right).
\]
Then, combining \eqref{eq:locscat0} and \eqref{eq:locscat}, it follows that
\begin{equation}
\label{eq:tanineq}
W_2(\mu_0,\mu_1)\le \|\phi(\mu_0)-\phi(\mu_1)\|_*,
\end{equation}
with equality if $\mu_0=\mu^*$ or $\mu_1=\mu^*$.
Therefore, since $\phi(M)$ is a subset of a vector space of dimension $D=d+d(d+1)/2$, there exists $C>0$ such that for all $\e>0$,
\[
N(B(\mu^*,r),W_2,\varepsilon)\le N(B((m_*,\Sigma^{1/2}_*),r)\cap\phi(M),\|.\|_*,\varepsilon)\le \left(\frac{Cr}{\varepsilon}\right)^{D}.
\]
Hence, \ref{B1} holds.
\end{exm}

The result, derived in example \ref{exm:lsfamilies}, may be generalised to other parametric subsets of the Wasserstein space (or more generally parametric subsets of geodesic Polish spaces with non-negative curvature).
Indeed, since the Wasserstein space over $\R^d$ has non-negative curvature, the support of $P$ pushed forward to the tangent cone at a barycenter $\mu^*$ is isometric to a Hilbert space (this result follows by combining Theorem \ref{thm:yokota} and Lemma \ref{lem:baryisexpbary}) and its norm satisfies \eqref{eq:tanineq} (see Proposition \ref{lem:tangentcone2}).
Therefore it is enough, for \ref{B1} to hold, to require that the image of the support of $P$ by the $\log_{\mu^*}$ map (see paragraph \ref{subsec:cone} for a definition) is included in a finite dimensional vector space.

\subsection{Polynomial metric entropy}
\label{subsec:pme}
Condition \ref{B1} is essentially a finite dimensional behaviour and does not apply in some scenarios of interest.
This paragraph addresses the situation where the complexity of set $\ms$, measured by its metric entropy, is polynomial.
\begin{enumerate}[label=(B2),leftmargin=*]
\item\label{B2} (Polynomial metric entropy) There exists constants $C,D>0$ such that, for all $\e>0$,
\[\log N(\ms,d,\e)\le \left(\frac{C}{\e}\right)^D.\]
\end{enumerate}

\begin{thm}
\label{thm:upperbound2}
Assume that \ref{A1}, \ref{A2}, \ref{A3} and \ref{B2} hold.
Then, for all $n\ge 1$ and all $t>0$,
\[
kd(x^*,x_n)^{\frac{2}{\beta}}\le  \int_M(F(x_n,y)-F(x^*,y))\,{\rm d}P(y) \le A\cdot\max\left\{v_n,\left(\frac{t}{n}\right)^{\frac{1}{2-\alpha\beta}}\right\},
\]
 with probability at least $1-2e^{-t}$, where 
 \[v_n=\left\{\begin{array}{lll}
 n^{-\frac{2}{4-(2\alpha-D)\beta}}&\mbox{ if } D<2\alpha,\\
(\log n)/\sqrt n&\mbox{ if } D=2\alpha,\\
n^{-\frac{\alpha}{D}}&\mbox{ if } D>2\alpha,\\
 \end{array}\right.\]
where $k=K^{-\frac{1}{\beta}}_3$ and $A$ is an explicit constant independent on $n$.
 \end{thm}

As for Theorem \ref{thm:upperbound1}, bounds in expectation may be easily derived from this result. The optimality of Theorem \ref{thm:upperbound2} is addressed in paragraph \ref{subsec:opt}.
Next is an example where assumption \ref{B2} applies.

\begin{exm}[Wasserstein space] 
\label{ex:covwass}
Let $B=\{x\in\R^d:\|x-x_0\|_2\le\rho\}$ be a closed Euclidean ball in $\R^d$ and let $\ms=\mathcal P_2(B)$ be the set of square-integrable probability measures supported on $B$ equipped with the 2-Wassertein metric $W_2$.
Combining the result of Appendix A in \cite{bolley2007} with a classical bound on the covering number of euclidean balls, it follows that for all $0<\e\le\rho$,
\[
\log N(M,W_2,\e)\le 2\left(\frac{6\rho}{\e}\right)^{d}\log\left(\frac{8e\rho}{\e}\right).
\]
In particular, for any $D>d$, there exists $C>0$ depending on $D$ and $\rho$ such that, for all $0<\e\le\rho$,
\[
\log N(M,W_2,\e)\le \left(\frac{C}{\e}\right)^{D},
\]
so that \ref{B2} is satisfied for all $D>d$.
\end{exm}

We finally point towards Theorem 3 in \cite{WeedBerthet19} which may be used to derive upper bounds on the covering number of subsets of the $2$-Wasserstein space composed of measures, absolutely continuous with respect to the Lebesgue measure, and with density belonging to some Besov class.

\subsection{On optimality}
\label{subsec:opt}

At the level of generality considered by Theorems \ref{thm:upperbound1} and \ref{thm:upperbound2}, we have not been able to assess the optimality of the given rates for all choices of functional $F$ satisfying the required assumptions and all values of constants $D,\alpha,\beta$.
In particular, it is likely that the rates displayed in Theorem \ref{thm:upperbound2} are artefacts of our proof techniques and that results may be improved in some specific scenarios using additional information on the problem at hand.
However, we discuss below some regimes where our results appear sharp and, on the contrary, settings where our results should allow for improvements.\\

To start our discussion, consider the barycenter problem, i.e. the case where $F(x,y)=d(x,y)^2$.
In the context where $(M,d)$ is a Hilbert space equipped with its usual metric and $P$ is square integrable, explicit computations reveal that  $x_n=\sum_{i=1}^nY_i/n$ is an empirical barycenter of $P$ in the sense of \eqref{eq:xn} and that $x^*=\esp[Y_1]$ (in the sense of the Pettis or Bochner integral) is the unique barycenter of $P$.
In addition, we check that, for all $n\ge 1$,
\begin{equation}
\label{eq:baryhilbert}
\esp d(x_n,x^*)^2=\esp \int_M(d(x_n,x)^2-d(x^*,x)^2)\,\mathrm{d}P(x)=\frac{1}{n}\int_Md(x,x^*)^2\,\mathrm{d}P(x).
\end{equation}
We notice that, under assumptions much more general than those considered in the present paper, the rate of convergence (in expectation) of empirical barycenters in a Hilbert space is of order $1/n$.
While this observation concerns the very special case of Hilbert spaces, we conjecture that the rates of convergence of empirical barycenters is of order $1/n$ in a wide family of metric spaces including Hilbert spaces as a special case.
Identifying precisely this wider family remains an open question but it appears from this discussion that boundedness and complexity restrictions, such as \ref{A1}, \ref{B1} and \ref{B2}, may be unnecessary for the barycenter problem.
Whenever $(M,d)$ is $\R^D$ equipped with a general norm, a very interesting recent contribution, connected to that question, is \cite{LuMe19}.
On a more positive note, we point towards two encouraging aspects encoded in our results in the context of the barycenter problem.
First, consider the case where $(M,d)$ is $\R^D$ equipped with the euclidean metric and suppose that the $Y_i$'s are independent with gaussian distribution $\mathcal N(x^*,\sigma^2)$.
Then, identity \eqref{eq:baryhilbert} reads in this case
\[
\esp d(x_n,x^*)^2=\esp \int_M(d(x_n,x)^2-d(x^*,x)^2)\,\mathrm{d}P(x)=\frac{\sigma^2D}{n}.
\]
It is known, furthermore, that $\sigma^2D/n$ corresponds (up to universal constants) to the minimax rate of estimation of $x^*$ in the context where the $Y_i$'s are i.i.d.
subgaussian random variables with mean $x^*$ and variance proxy $\sigma^2$ (see Chapter 4 in \cite{Rig17}).
Therefore, provided $(M,d)$ is a bounded metric space (which guarantees \ref{A1} and \ref{A2} with $\alpha=1$) and provided assumption \ref{A3} holds for $\beta=1$ (which is often the case as discussed in paragraphs \ref{subsec:negcurv} and \ref{subsec:poscurv} below) Theorem \ref{thm:upperbound1} recovers the optimal rate of convergence $D/n$, up to constants, in a fairly wide context.
Finally, note that while possibly suboptimal in some cases, the rates provided by Theorems \ref{thm:upperbound1} and \ref{thm:upperbound2}, combined with examples \ref{exm:lsfamilies}, \ref{ex:covwass} and discussions of paragraph \ref{subsec:poscurv}, provide up to our knowledge the first rates for the Wasserstein barycenter problem at this level of generality.
An exception is the Wasserstein space over the real line (studied, for instance, in \cite{Bigotandco2018}) which happens to be isometric to a convex subset of a Hilbert space as can be deduced for instance from combining statement (iii) of Proposition 3.5 in \cite{sturm2003} and Proposition 4.1 in \cite{kloeckner2010}.\\ 

Outside from the setting of the barycenter problem, not much is known on optimal rates of estimation or learning (in the sense described at the end of paragraph \ref{subsec:setup}) of $x^*$ defined in \eqref{eq:xstar}.
We believe this question remains mainly open.
It is our impression that the $1/n$ rate, conjectured to hold for empirical barycenters in a wide setup, is a behavior very specific to the case $F(x,y)=d(x,y)^2$.
For more general functionals, we suspect that the complexity of $M$ should have an impact as it is classically the case in nonparametric statistics or learning theory.
Note in particular that whenever parameters $\alpha=\beta=1$ in \ref{A2} and \ref{A3}, the rate $v_n$ in Theorem \ref{thm:upperbound2} becomes
 \[v_n=\left\{\begin{array}{lll}
 n^{-\frac{2}{2+D}}&\mbox{ if } D<2,\\
(\log n)/\sqrt n&\mbox{ if } D=2,\\
n^{-\frac{1}{D}}&\mbox{ if } D>2,\\
 \end{array}\right.\]
which corresponds to known state of the art learning rates, under complexity assumptions in the same flavor as \ref{B2}, as displayed for instance by Theorem 2 in \cite{RaSrTs17}.
However, exact situations under which Theorem \ref{thm:upperbound2} provides optimal rates of convergence remains unclear to us.\\

Finally, note that the second inequality in both Theorems \ref{thm:upperbound1} and \ref{thm:upperbound2} hold for the limiting case $\beta=0$ (with $A$ remaining finite), which correspond to dropping assumption \ref{A3}.
In the context of Theorem \ref{thm:upperbound1} (or that of Theorem \ref{thm:upperbound2} with $D<2\alpha$) the case $\beta=0$ gives rise to a bound of order \[\int_\ms (F(x_n,y)-F(x^*,y))\,{\rm d}P(y)\le \frac{C}{\sqrt n},\] 
with high probability.
Note however that this limiting case does not allow to provide any bound for $d(x_n,x^*)$.

\section{Variance inequalities}
\label{sec:vi}

This section studies conditions implying the validity of \ref{A3}.
The first three paragraphs below focus on the barycenter problem, i.e. the case where $F(x,y)=d(x,y)^2$, and investigate \ref{A3} in the light of curvature bounds.
Aleksandrov curvature bounds of a geodesic space (see paragraph \ref{subsec:curvature}) is a key concept of comparison geometry and many geometric phenomena are known to depend on whether the space $M$ has a curvature bounded from below or above.
In paragraphs \ref{subsec:negcurv} and \ref{subsec:poscurv}, devoted respectively to non-positively and non-negatively curved spaces, we show that curvature bounds also affect statistical procedures through their relation with \ref{A3}.
Finally, paragraph \ref{subsec:conv} addresses the case of a general $F$ and connects \ref{A3} to its convexity properties.
The material presented in this section relies heavily on background in metric geometry gathered in appendix \ref{app:mg} for convenience.

\subsection{Non positive curvature}
\label{subsec:negcurv}

This first paragraph introduces a fundamental insight due to K.T. Sturm, in the context of geodesic spaces of non-positive curvature, that has strongly influenced our study. To put the following result in perspective, we recall that a geodesic space $(M,d)$ is said to have non-positive curvature ($\mathrm{curv}(M)\le 0$ for short) if, for any $p,x,y\in M$ and any geodesic $\gamma:[0,1]\to M$ such that $\gamma(0)=x$ and $\gamma(1)=y$,
\[d(p,\gamma(t))^2\le (1-t)d(p,x)^2+td(p,y)^2-t(1-t)d(x,y)^2,\]
for all $t\in[0,1]$.
Non-positive curvature is given a probabilistic description in the next result.

 \begin{thm}[Theorem 4.9 in \cite{sturm2003}]
 \label{thm:sturm}
 Let $(\ms,d)$ be a separable and complete metric space.
Then, the following properties are equivalent.
 \begin{itemize}
     \item[$(1)$] $(\ms,d)$ is geodesic and ${\rm curv}(\ms)\le 0$.
     \item[$(2)$] Any probability measure $Q\in \mathcal P_2(\ms)$ has a unique barycenter $x^{*}\in \ms$ and, for all $x\in \ms$, 
    \begin{equation}
        \label{eq:vistrurm}
    d(x,x^*)^2\le \int_{\ms}(d(x,y)^2-d(x^*,y)^2)\,{\rm d}Q(y).
    \end{equation}
 \end{itemize}
 \end{thm}
In words, Theorem \ref{thm:sturm} states in particular that \ref{A3} holds for any possible probability measure $P$ on $M$, with $K_3=1$ and $\beta=1$, provided $\mathrm{curv}(M)\le0$.
It is worth mentioning again that \ref{A1} and \ref{A2} also hold, provided in addition $\mathrm{diam}(M)<+\infty$, so that the case of bounded metric spaces with non-positive curvature fits very well our basic assumptions.
Condition $\mathrm{curv}(M)\le0$ is satisfied in a number of interesting examples.
Such examples include (convex subsets of) Hilbert spaces or the case where $(M,d)$ is a simply connected Riemannian manifold with non-positive sectional curvature. Other examples are metric trees and other metric constructions such as products or gluings of spaces of non-positive curvature (see \cite{BriHaf99}, \cite{BuBuIv01} or \cite{AlKaPe17} for more details).

\subsection{Non negative curvature and extendable geodesics}
\label{subsec:poscurv}
The present paragraph investigates the case of spaces of non-negative curvature.
Contrary to the case of spaces of non-positive curvature, condition \ref{A3} may not hold for every probability measure $P$ on $M$ if $\curv(M)\ge 0$.
Indeed, note that unlike in the case when $\mathrm{curv}(M)\le 0$, there might exist probability measures $P\in\mathcal P_2(M)$ with more than one barycenter whenever $\mathrm{curv}(M)\ge 0$.
A simple example when $M=S^{d-1}$, the unit euclidean sphere in $\R^d$ with angular metric, is the uniform measure on the equator having the north and south poles as barycenters.
Since \ref{A3} implies uniqueness of barycenter $x^*$, this condition disqualifies such probability measures.
Hence, establishing conditions under which \ref{A3} holds is more delicate whenever $\mathrm{curv}(M)\ge 0$.
The next result provides an important first step in this direction. 
\begin{thm}
\label{thm:varequal}
Let $(\ms,d)$ be a separable and complete geodesic space such that ${\rm curv}(\ms)\ge 0$.
Let $P\in\mathcal P_2(\ms)$ and $x^*$ be a barycenter of $P$. Then, for all $x\in \ms$,
\begin{equation}
    \label{eq:vik}
    d(x,x^*)^2 \int_{\ms}k^x_{x^*}(y)\,{\rm d}P(y)= \int_{\ms}(d(x,y)^2-d(x^*,y)^2)\,{\rm d}P(y),
\end{equation}
where, for all $x\ne x^*$ and all $y$, 
\begin{equation}
\label{eq:k}
k^x_{x^*}(y)=1-\frac{\|\log_{x^*}(x)-\log_{x^*}(y)\|^2_{x^*}-d(x,y)^2}{d(x,x^*)^2}.
\end{equation}
Therefore, $P$ satisfies \ref{A3} with $K_3=1/k$ and $\beta=1$ if and only if, for all $x\in M$,
\begin{equation}
\label{eq:kvi}    
k\le \int_{\ms}k^x_{x^*}(y)\,{\rm d}P(y).
\end{equation}
\end{thm}

By definition of a barycenter, the right hand side of \eqref{eq:vik} is non negative.
In addition $\mathrm{curv}(M)\ge 0$ implies that $d(x,y)\le\|\log_p(x)-\log_p(y)\|_p$ for all $x,y,p\in \ms$ (see Proposition \ref{lem:tangentcone2} in appendix \ref{app:mg}).
Combining these two observations with the definition of $k_{x^*}^x(y)$ implies that  
\[
0\le \int_{\ms}k^x_{x^*}(y)\,{\rm d}P(y)\le 1.
\]

The next result identifies a condition under which a variance inequality holds.

\begin{thm}\label{thm:extendgeod}
Let $(\ms,d)$ be a separable and complete geodesic space such that ${\rm curv}(\ms)\ge 0$.
Let $P\in\mathcal P_2(\ms)$ and $x^*$ be a barycenter of $P$.
Fix $\lambda>0$ and suppose that the following properties hold.
\begin{itemize}
	\item[$(1)$] For $P$-almost all $y\in M$, there exists a geodesic $\gamma_y:[0,1]\to \ms$ connecting $x^*$ to $y$ that can be extended to a function $\gamma^{+}_{y}:[0,1+\lambda]\to \ms$ that remains a shortest path between its endpoints. \label{eq:extcond1}
	\item[$(2)$] The point $x^*$ remains a barycenter of the measure $P_{\lambda}=(e_{\lambda})_{\#}P$ where $e_{\lambda}:M\to M$ is defined by $e_{\lambda}(y)=\gamma^{+}_y(1+\lambda)$. \label{eq:extcond2}
\end{itemize}
Then, for all $x \in M$, 
\[d(x,x^*)^2\le \frac{1+\lambda}{\lambda} \int_{\ms}(d(x,y)^2-d(x^*,y)^2)\,{\rm d}P(y),\]
and thus \ref{A3} holds with $K_3=(1+\lambda)/\lambda$ and $\beta=1$.
\end{thm}

Examples of geodesic spaces of non-negative curvature include (convex subsets of) Hilbert spaces or the case where $(M,d)$ is a simply connected Riemannian manifold with non-negative sectional curvature.
Next is a simple example where the condition $(1)$ of Theorem \ref{thm:extendgeod}, i.e. the ability to extend geodesics, takes a simple form.
\begin{exm}[Unit sphere]
	\label{ex:unitsph}
Let $M=S^{d-1}$ be the unit Euclidean sphere in $\R^d$ equipped with the angle metric.
Let $P\in\mathcal P_2(M)$ be such that it has a unique barycenter $x^*\in M$.
In $M$, a shortest path between two points is a part of a great circle and a part of a great circle is a shortest path between its endpoints if, and only if, it has length less than $\pi$.
Therefore, if a neighborhood $V$ of $C(x^*)=\{-x^*\}$, the cut locus of $x^*$, satisfies $P(V)=0$, then condition $(1)$ of Theorem \ref{thm:extendgeod} is satisfied for some $\lambda>0$.
Note however that condition $(1)$ is not enough to give a variance inequality in general.
Indeed, consider the uniform measure $P$ on the equator with the north and south poles for barycenters.
Then, for $P$-almost all $y\in M$, the geodesic connecting the south pole $x^*$ to $y$ can be extended by a factor $1+\lambda=2$ in the sense of $(1)$ in the above theorem.
However, since there is no unique barycenter, no variance inequality can hold in this case.
Therefore, requirement $(2)$ cannot be dropped in Theorem \ref{thm:extendgeod}.
\end{exm}

In the rest of this paragraph, we provide a sufficient condition for the extendable geodesics condition $(1)$ of Theorem \ref{thm:extendgeod} to hold in the context where $(M,d)=(\mathcal P_2(H),W_2)$ is the Wasserstein space over a Hilbert space $H$.
In this case, $(M,d)$ is known to have non-negative curvature (see Section 7.3 in \cite{ambrosio2008gradient}).
We recall the following definition.
For a convex function $\phi:H\to \R$, its subdifferential $\partial\phi\subset H^2$ is defined by 
\[
(x,y)\in\partial\phi \Leftrightarrow \forall z\in H,\quad\phi(z)\ge\phi(x)+\langle y,z-x\rangle.
\]
Then we can prove the following result.
\begin{thm}
\label{thm:extendgeodp2}
Let $(M,d)=(\mathcal P_2(H),W_2)$ be the Wasserstein space over a Hilbert space $H$.
Let $\mu$ and $\nu$ be two elements of $M$ and let $\gamma:[0,1]\to S$ be a geodesic connecting $\mu$ to $\nu$ in $M$.
Then, $\gamma$ can be extended by a factor $1+\lambda$ (in the sense of $(1)$ in Theorem \ref{thm:extendgeod}) if, and only if, the support of the optimal transport plan $\pi$ of $(\mu,\nu)$ lies in the subdifferential $\partial\phi$ of a $\frac{\lambda}{1+\lambda}$-strongly convex map $\phi:H\to \R$.
\end{thm}
Again, like in example \ref{ex:unitsph}, this gives a condition that ensures the validity of $(1)$ in Theorem \ref{thm:extendgeodp2}, but it can be shown that it is not enough to obtain a variance inequality.

\subsection{Convexity}
\label{subsec:conv}

Here, we connect \ref{A3} to convexity properties of functional $F$ along paths in $\ms$.

\begin{defi}[$(k,\beta)$-convexity]
Given $k>0$, $\beta\in(0,1]$ and a path $\gamma:[0,1]\to\ms$, a function $f:\ms\to \R$ is called $(k,\beta)$-convex along $\gamma$ if the function
\[t\in[0,1]\mapsto f(\gamma(t))-kd(\gamma(0),\gamma(1))^{\frac{2}{\beta}}t^2\]
is convex.
If $(M,d)$ is geodesic, a function $f:\ms\to \R$ is called  $(k,\beta)$-geodesically convex if, for all $x,x'\in\ms$, $f$ is $(k,\beta)$-convex along at least one geodesic connecting $x$ to $x'$.
\end{defi}

In the sequel, we abbreviate $(k,1)$-convexity by $k$-convexity.
When $M$ is geodesic, a $(k,\beta)$-convex function $f:M\to\R$ refers to a $(k,\beta)$-geodesically convex function unless stated otherwise.
Note that $(k,\beta)$-convexity is a special case of uniform convexity (see Definition 1.6. in \cite{sturm2003}). We start by a general result.
\begin{thm}
\label{thm:fvi}
Let $k>0$ and $\beta\in(0,1]$.
Suppose that, for all $x\in \ms$, there exists a path connecting $x^*$ to $x$ along which the function 
\[x\in \ms\mapsto \int_{M} F(x,y)\,{\rm d}P(y)\]
is $(k,\beta)$-convex.
Then, for all $x\in \ms$,
\begin{equation}
\label{eq:gvi}
d(x,x^*)^2\le \left(\frac{1}{k}\int_{M}(F(x,y)-F(x^*,y))\,{\rm d}P(y)\right)^{\beta},
\end{equation}
and hence \ref{A3} holds for $K_3=k^{-\beta}$.
In particular, the assumption of the theorem holds whenever, for all $x\in \ms$ and ($P$-almost) all $y\in M$, there exists a path connecting $x^*$ to $x$ along which the function $F(.,y)$ is itself $(k,\beta)$-convex.
A special case in when $(M,d)$ is geodesic and, for ($P$-almost) all $y\in M$, the functional $x\in\ms\mapsto F(x,y)$ is itself $(k,\beta)$-geodesically convex.
\end{thm}

The previous result is deliberately stated in a general form.
This allows to investigate several notions of convexity that may coexists as in the space of probability measures.

\begin{rem}
\label{rem:convvi}
For some spaces, such as the Wasserstein space over $\R^d$, there exist two canonical paths between two points $\mu$ and $\nu$.
One is the geodesic $\gamma_t$ defined by the fact that $W_2(\gamma_t,\gamma_s)=|s-t|W_2(\mu,\nu)$ for $s,t\in[0,1]$ (which needs not be unique).
A second one is the linear interpolation between the two measures $\ell_t=(1-t)\mu+t\nu$.
It may be that $x\mapsto \int F(x,y){\rm d}P(y)$ is strongly convex along only one of these two paths.
This case is further discussed in section \ref{sec:exmps}.
\end{rem}

We end by discussing the notion of a $k$-convex metric space.
In the remainder of the paragraph, we consider $F(x,y)=d(x,y)^2$.
   
\begin{defi}
\label{def:kconv}
A geodesic space $(\ms,d)$ is said to be $k$-convex if, for all $y\in\ms$, the function $x\in\ms\mapsto d(x,y)^2$ is $k$-geodesically convex.
\end{defi}
\noindent Using Theorem \ref{thm:fvi}, it follows that if $\ms$ is $k$-convex, then \ref{A3} holds with $\beta=1$ and $K_3=1/k$.
When $k=1$, note that a geodesic space is $k$-convex if and only if ${\rm curv}(M)\le 0$ (see Proposition \ref{pro:NNC}).
When $k\ne 1$, the connection between curvature bounds and $k$-convexity of a metric space is not straightforward. In particular, $k$-convex metric spaces include many interesting spaces, for which condition ${\rm curv}(\ms)\le 0$ does not hold. We give two examples.

\begin{exm}[Proposition 3.1. in \cite{ohta2007}]
\label{thm:ohtastyle}
Let $\ms$ be a geodesic space with $\mathrm{curv}(\ms)\le 1$ and $\mathrm{diam}(\ms)<\pi/2$.
Then $\ms$ is $k$-convex for 
\[
k=4\,\mathrm{diam}(\ms)\tan\left(\frac{\pi}{2}-\mathrm{diam}(\ms)\right).
\]
Implication of this result can be compared to Theorem \ref{thm:extendgeod}.
In the context of $M\subset S^{2}$, considered in example \ref{ex:unitsph}, the above result states that $M$ is $k$-convex provided it is included in the interior of an $1/8$th of a sphere.
In comparison, Theorem \ref{thm:extendgeod} expresses that a variance inequality \ref{A3} may hold for a measure supported on the whole sphere minus a neighbourhood of the cut locus of the barycenter.

\end{exm}

\begin{exm}[Theorem 1 in \cite{BCL94}] If $(S,\mathcal S,\mu)$ is a measured space and $p\in(1,2]$, then $M=L^p(\mu)$ is $(p-1)$-convex.
\end{exm}

\section{Further examples}
\label{sec:exmps}

We describe additional examples, different from the usual barycenter problem, where our results apply.
In these examples, we focus mainly on functionals over the Wasserstein space.
Paragraphs \ref{subsec:fdiv} addresses the case of $f$-divergences.
Subsection \ref{subsec:interenergy} discusses interaction energy.
In \ref{subsec:approxbary} we consider several examples related to the approximation of barycenters.

\subsection{\texorpdfstring{$f$-divergences}{f-divergence}}
\label{subsec:fdiv}
We call $f$-divergence a functional of the form
\begin{equation}
\label{eq:relativef}
D_f(\mu,\nu):=\left\{\begin{array}{ll}
\int f\left(\frac{{\rm d}\mu}{{\rm d}\nu}\right)\,{\rm d}\nu &\mbox{ if }\mu\ll \nu,\\
+\infty&\mbox{ otherwise,}\
\end{array}\right.
\end{equation}
for some convex function $f:\R_+\to \R$.
Such functionals are also known as internal energy, relative functionals or Csisz\'ar divergences.
Specific choices for function $f$ give rise to well known examples.
For instance, $f(x)=x\log x$ gives rise to the Kullback-Leibler divergence (or relative entropy).
Minimisers of the average Kullback-Leibler divergence, or its symmetrised version, have been considered for instance in \cite{veldhuis2002centroid} for speech synthesis.
Other functions like $f(x)=(x-1)^2$ or $f(x)=|x-1|/2$ lead respectively to the chi-squared divergence or the total variation.
The next results present sufficient conditions under which \ref{A1}, \ref{A2} and \ref{A3} hold in this case.
First, suppose $M\subset \mathcal P_2(E)$ where $E\subset \R^d$ is a convex set.
Note that \ref{A1} holds if there exists $0<c_-<c_+<+\infty$ and a reference measure such that all $\mu\in M$ have a density $g_{\mu}$ with respect to this measure with values in $[c_-,c_+]$ on their support.
Next, we show that \ref{A2} holds under related conditions.

\begin{thm}
\label{thm:dfA2}
Suppose that all measures $\mu\in M$ have a density $g_{\mu}$ with respect to some reference measure.
Suppose there exist $0<c_-<c_+<+\infty$ such that all $g_{\mu}$ take values in $[c_-,c_+]$.
Suppose in addition that there exists $\Lambda>0$ such that all $g_\mu$ are $\Lambda$-Lipschitz on $E$.
Assume finally that $f$ is differentiable and that $f'$ is $L$-Lipschitz on $E$.
Then, for all $\nu,\mu,\mu'\in M$,
\[|D_f(\mu,\nu)-D_f(\mu',\nu)|\le \frac{2L\Lambda c_+}{c^2_-}W_2(\mu,\mu'),\]
so that \ref{A2} holds for $K_2=2L\Lambda c_+/c^2_-$ and $\alpha=1$.
\end{thm}
 The next result is devoted to condition \ref{A3} and is a slight modification of Theorem 9.4.12 in \cite{ambrosio2008gradient}.
Below we say that $\nu\in M$ is $\lambda$-log-concave if $\mathrm{d}\nu(x)=e^{-U(x)}\mathrm{d}x$ for some $\lambda$-convex $U:E\to \R$.
\begin{thm}
\label{thm:dfA3}
Consider $h(s)=e^{s}f(e^{-s})$.
Suppose that $h$ is convex and that there exists $c>0$ such that, for all $u\in \R$ and all $\e>0$, $h(u)-h(u+\e)\ge c\e$.
Suppose finally that $M$ is a geodesically convex subset of $\mathcal P_2(E)$.
Then the following holds.
\begin{itemize}
    \item[$(1)$] For any $\lambda$-log-concave $\nu\in M$, the functional $\mu\in M\mapsto D_f(\mu,\nu)$ is $c\lambda$-geodesically convex.
    \item[$(2)$] Let $\lambda>0$ and suppose $P\in \mathcal P(M)$ is supported on $\lambda$-log-concave measures in $M$.
Then for any minimizer 
    \[\mu^*\in\underset{\mu\in M}{\arg\min}\int_M D_{f}(\mu,\nu)\,\mathrm{d}P(\nu),\]
    and any $\mu\in M$,
    \[W_2(\mu,\mu^*)^2\le\frac{1}{c\lambda}\int_M(D_f(\mu,\nu)-D_f(\mu^*,\nu))\,\mathrm{d}P(\nu),\]
    and thus \ref{A3} holds with $K_3=1/c\lambda$ and $\beta=1$.
\end{itemize}
\end{thm}

Note that the requirements made on $f$ in Theorem \ref{thm:dfA3} are compatible with $f(x)=x\log x$ for $c=1$.
However, these requirements exclude examples such as $f(x)=(x-1)^2$.
At the price of a smaller exponent $\beta=1/2$, the next result shows that \ref{A3} holds for any choice of a strongly convex $f$. 

\begin{thm}
\label{thm:dflinear}
Let $f:\R_+\to\R$ be $k$-convex.
Then the following holds. 
\begin{itemize}
    \item[$(1)$] Let $\nu,\mu_0,\mu_1\in\mathcal P_2(E)$ be such that $\mu_0\ll \nu$ and $\mu_1\ll\nu$.
Suppose in addition that 
\[m_4(\nu):=\int \|x\|^4\,{\rm d}\nu(x)<+\infty.\]
Then, $D_f(.,\nu)$, is $(\frac{k}{4m_4(\nu)},\frac{1}{2})$-convex along the linear interpolation $\ell_t=(1-t)\mu_0+t\mu_1$.
\item[$(2)$] Let $m>0$.
Suppose that $M$ is a linearly convex subset of $\mathcal P_2(E)$ and suppose $P\in \mathcal P(M)$ is supported on measures $\nu\in M$ such that $m_4(\nu)\le m$.
Then, for any minimizer 
    \[\mu^*\in\underset{\mu\in M}{\arg\min}\int_M D_{f}(\mu,\nu)\,\mathrm{d}P(\nu),\]
    and any $\mu\in M$,
    \[W_2(\mu,\mu^*)^2\le2\left(\frac{m}{k}\int_M(D_f(\mu,\nu)-D_f(\mu^*,\nu))\,\mathrm{d}P(\nu)\right)^{1/2},\]
    and thus \ref{A3} holds with $K_3=2(m/k)^{1/2}$ and $\beta=1/2$.
\end{itemize}
\end{thm}

We end this paragraph with an important remark.
Using Theorem \ref{thm:fvi}, and as in Theorem \ref{thm:dfA3}, condition \ref{A3} can be deduced from $k$-convexity of the $f$-divergence.
Strong convexity of the $f$-divergence was used by Sturm, Lott and Villani (see \cite{Sturm06I,LottVillani2009} and references therein) to define the celebrated synthetic notion of Ricci curvature bounds.
In particular, for an underlying Riemannian $D$-dimensional manifold $E$ equipped with a measure with density $e^{-V}$ with respect to the Riemannian measure, $k$-convexity of such $f$-divergence is equivalent to 
\[
\Ric+\Hess(V)\ge k
\]
where $\Ric$ stands for the Ricci curvature tensor and $\Hess(V)$ is the Hessian of $V$, given $s\mapsto s^Df(s^{-D})$ is geodesically convex.

\subsection{Interaction energy}
\label{subsec:interenergy}
Interaction energy of a measure $\mu$, for a given function $g$ on $M\times M$ called interaction potential, is usually defined as the integral of $g$ w.r.t. the product measure $\mu\otimes\mu$.
Here, we keep the same terminology for a larger class of functionals on the Wasserstein space. 
Let $(E,\rho)$ be a separable, complete, locally compact and geodesic space.
Let $M=\mathcal P_2(E)$ be equipped with the Wasserstein metric $W_2$.
Consider functional $I_g:M\times M\to \R_+$ defined by 
\[I_g(\mu,\nu)=\int_{E\times E} g(x,y)\,\mathrm{d}\mu(x)\mathrm{d}\nu(y),\]
for a measurable $g:E\times E\to\R_+$.
Note first that \ref{A1} is clearly satisfied provided $g$ is upper bounded or $g$ is continuous and $E$ is compact.
The next two results present sufficient conditions for \ref{A2} and \ref{A3} to hold.
\begin{thm}
\label{thm:interactionA2}
Fix $L>0$ and suppose that $g$ is $L$-Lipschitz in the first variable when the second is fixed.
Then, for all $\mu,\mu',\nu\in\mathcal P_2(E)$,
\[
|I_g(\mu,\nu)-I_g(\mu',\nu)|\le L W_2(\mu,\mu'),
\]
and \ref{A2} holds with $K_2=L$ and $\alpha=1$.
\end{thm}
\begin{thm}
\label{thm:interactionA3}
Fix $k\ge0$ and $\beta\in(0,1]$.
Suppose that $g$ is $(k,\beta)$-geodesically convex in the first variable when the second is fixed.
Then, the following holds.
\begin{itemize}
    \item[$(1)$] For any fixed $\nu\in\mathcal P_2(E)$, the functional $I_g(.,\nu)$ is $(k,\beta)$-geodesically convex on $\mathcal P_2(E)$.
    \item[$(2)$] Let $P\in\mathcal P(\mathcal P_2(E))$.
Then for any minimizer 
    \[
    \mu^*\in\underset{\mu\in\mathcal P_2(E)}{\arg\min}\int I_{g}(\mu,\nu)\,\mathrm{d}P(\nu),
    \]
    and any $\mu\in\mathcal P_2(E)$,
    \[
    W_2(\mu,\mu^*)^2\le\left(\frac{1}{k}\int(I_g(\mu,\nu)-I_g(\mu^*,\nu))\,\mathrm{d}P(\nu)\right)^{\beta},
    \]
    and \ref{A3} holds with $K_3=1/k^{\beta}$.
    \end{itemize}
\end{thm}

\subsection{Regularised Wasserstein distance}
\label{subsec:approxbary}

In this final paragraph, we present a few functionals of interest whose study, in the light of the previous results, could provide interesting research perspectives.
In certain cases, like for the Wasserstein space, the distance is difficult to compute.
In the field of computational optimal transport, a lot of work has been devoted to construct computationally friendly approximations of the Wasserstein distance.
Such approximations include the sliced Wasserstein distance, penalised Wasserstein distances or regularised Wasserstein distances (see \cite{cp2018} for more details).
In some of these approximations, one hope is that one may enforce some form of convexity allowing for condition \ref{A3} to be valid in a wide setting.
The Sinkhorn divergence, defined below, is an approximation of the Wasserstein distance that has been widely used in practice due to its attractive computationally properties.
Denote by $D$ the relative entropy, i.e. $D=D_f$ for $f(x)=x\log x$ where $D_f$ is as in \eqref{eq:relativef}, and $\Gamma(\mu,\nu)$ the set of probability measures $\pi$ on $\R^d\times\R^d$ with marginals $\mu$ an $\nu$ respectively.
Let $\mu,\nu\in\mathcal{P}_2(\R^d)$, and $\gamma>0$.
The entropy-regularised Wasserstein distance or Sinkhorn divergence between $\mu$ and $\nu$ is defined by
\begin{equation}
\label{eq:sinkh}
W_\gamma(\mu,\nu)^2:=\inf_{\pi\in\Gamma(\mu,\nu)}\int d(x,y)^2\,{\rm d}\pi(x,y) + \gamma D(\pi,\mu\otimes\nu).
\end{equation}
Identifying scenarios in which $W_\gamma$ satisfies conditions \ref{A1}-\ref{A3} remains for us an open question. 
Note however that, using Lagrange multipliers, the Sinkhorn divergence can be alternatively expressed as
\[
\inf\left\{\int d(x,y)^2\,{\rm d}\pi(x,y):\pi\in\Gamma(\mu,\nu),D(\pi,\mu\otimes\nu)\le\delta\right\},
\]
for some $\delta\ge 0$.
Letting $\delta=0$, the above expression reduces to the interaction energy $I_g$ studied in paragraph \ref{subsec:interenergy} for $g(x,y)=d(x,y)^2$.
This observation suggests that, for large values of $\gamma$, functional $W_\gamma(\mu,\nu)$ should satisfy \ref{A1}-\ref{A3} under reasonable conditions. 
An interesting alternative notion of regularised Wasserstein distance, via factored couplings, was introduced in \cite{forrow2019statistical}.
In order to define this notion, we introduce some notation.
For a partition $\mathcal C=(C_1,\dots,C_n)$, and a measure $\mu$, denote by $\mu_i^\mathcal{C}$ the measure $\mu$ restricted to $C_i$.
For $n\in \N$, denote by $\Gamma_n(\mu,\nu)$ the set of measures $\gamma$ such that there exists two partition $\mathcal{C}^0$ and $\mathcal{C}^1$ such that $\gamma=\sum_{i=1}^n\lambda_i\mu_i^{\mathcal C^0}\otimes\nu_i^{\mathcal C^1}$, with $\lambda_i:=\mu(C_i^0)=\nu(C_i^1)$.
For two measures $\mu$ and $\nu$, the regularised Wasserstein distance between $\mu$ and $\nu$ is defined by
\[
RW_n^2(\mu,\nu):=\inf_{\gamma\in\Gamma_n(\mu,\nu)} \int d^2 {\rm d}\gamma=\inf\left\{\sum_{i=1}^n \lambda_i \int d^2{\rm d}(\mu_i\otimes\nu_i): \sum_{i=1}^n\lambda_i\mu_i\otimes\nu_i=\gamma\in\Gamma_n(\mu,\nu)\right\}.
\]
While the trivial case $n=1$ reduces to the interaction energy studied above, it remains unclear to us in which setting our assumptions apply to this functional in general.

\section{Proofs}
\label{sec:proofs}
\subsection{Proof of Theorem \ref{thm:upperbound1}}

The proof is based on three auxiliary lemmas.
The first lemma provides an upper bound on the largest fixed point of a random nonnegative function.
The proof follows from a combination of arguments presented in Theorem 4.1, Corollary 4.1 and Theorem 4.3 in \cite{Kolt11}.
\begin{lem}
\label{lem:fixedpoint}
Let $\{\phi(\delta):\delta\ge 0\}$ be non-negative random variables (indexed by all deterministic $\delta\ge 0$) such that, almost surely, $\phi(\delta)\le \phi(\delta')$ if $\delta\le\delta'$.
Let $\{b(\delta,t): \delta\ge 0, t\ge0\}$, be (deterministic) real numbers such that $b(\delta,t)\le b(\delta,t')$, as soon as $t\le t'$, and such that
\[\prob(\phi(\delta)\ge b(\delta,t))\le e^{-t}.\]
Finally, let $\hat \delta$ be a nonnegative random variable, a priori upper bounded by a constant $\bar\delta>0$, and such that, almost surely, 
\[\hat \delta\le \phi(\hat \delta).\]
Then defining, for all $t\ge0$,
\[b(t):=\inf\left\{\alpha>0:\sup_{\delta\ge \alpha}\frac{b\left(\delta,\tfrac{t\delta}{\alpha}\right)}{\delta}\le 1\right\},\]
we obtain, for all $t\ge0$, 
\[\prob(\hat \delta\ge b(t))\le 2e^{-t}.\]
\end{lem}

The second lemma is due to \cite{Bous02} and improves upon the work of \cite{Tala96} by providing explicit constants.
Given a family $\mathcal F$ of functions $f:M\to \R$, denote 
\[| P-P_{n}|_{\mathcal F}:=\sup_{f\in\mathcal F}(P-P_n)f\quad\mbox{and}\quad \sigma^{2}_{\mathcal F}:=\sup_{f\in\mathcal F} P(f-Pf)^2,\]
where $M$ and $P$ are as defined in paragraph \ref{subsec:setup} and $P_n=\frac{1}{n}\sum_{i=1}^n\delta_{Y_i}$.
\begin{lem}
\label{bousquet}
Suppose that all functions in $\mathcal F$ are $[a, b]$-valued, for some $a<b$.
Then, for all $n\ge 1$ and all $t>0$, 
\[|P-P_{n}|_{\mathcal F}\le \esp |P-P_{n}|_{\mathcal F}+\sqrt{\frac{2t}{n}\left(\sigma^{2}_{\mathcal F}+2(b-a)\,\esp |P-P_{n}|_{\mathcal F}\right)}+\frac{(b-a)t}{3n},\]
with probability larger than $1-e^{-t}$.
\end{lem}
For background on empirical processes, including the proof of Lemma \ref{bousquet}, we refer the reader to  \cite{GineNick15}.
Finally, the third result we need is the following generalized version of Dudley's entropy bound (see, for instance, Theorem 5.31 in \cite{vanHa16}).

\begin{lem}
\label{lem:chain}
Let $(X_t)_{t\in E}$ be a real valued process indexed by a pseudo metric space $(E,d)$.
Suppose that the three following conditions hold.\vspace{0.2cm}\\
$(1)$ \emph{(Separability)} There exists a countable subset $E'\subset E$ such that, for any $t\in E$,
\[X_{t}=\lim_{t'\to t, t'\in E'}X_{t'},\quad \mbox{a.s.}\] 
$(2)$ \emph{(Subgaussianity)} For all $s,t\in E$, $X_{s}-X_{t}$ is subgaussian in the sense that
\[\forall \theta\in\R,\quad \log\esp e^{\theta(X_s-X_t)}\le \frac{\theta^2d(s,t)^2}{2}.\]
$(3)$ \emph{(Lipschitz property)} There exists a random variable $L$ such that, for all $s,t\in E$,
\[|X_s-X_t|\le Ld(s,t),\quad\mbox{a.s.}\]
Then, for any $S\subset E$ and any $\e\ge 0$, we have
$$\esp\sup_{t\in S} X_{t}\le 2\e\esp[L]+ 12\int_{\e}^{+\infty}\sqrt{\log N(S,d,u)}\,\emph{d}u.$$
\end{lem}

We are now in position to prove Theorem \ref{thm:upperbound1}.
 
\begin{proof}[Proof of Theorem \ref{thm:upperbound1}] 
$(1)$ For any $\delta\ge 0$, denote 
\[\ms(\delta):=\{x\in\ms: P(F(x,.)-F(x^*,.))\le \delta\},\]
and
\[\phi_n(\delta):=\sup\{(P-P_n)(F(x,.)-F(x^*,.)):x\in\ms(\delta)\}.\]
As a consequence of Assumption \ref{B1}, the set $\ms$ is separable.
Hence, the quantity $\phi_n(\delta)$ is measurable, as well as all suprema involved in the rest of the proof.
Define 
\[\delta_n:= P(F(x_n,.)-F(x^*,.)).\]
By definition of $x_n$, $P_n(F(x_n,.)-F(x^*,.))\le 0$ so that
\begin{equation}
\nonumber
\delta_n \le (P-P_n)(F(x_n,.)-F(x^*,.))\le \phi_n(\delta_n).
\end{equation}
As a result, in order to upper bound $\delta_n$ with high probability, it is enough to upper bound $\phi_n(\delta)$, for fixed $\delta\ge 0$, and apply Lemma \ref{lem:fixedpoint}.
Denoting
\[\sigma^2(\delta):=\sup\{P(F(x,.)-F(x^*,.))^2:x\in\ms(\delta)\},\]
 and observing that $-2K_1\le F(x,y)-F(x^*,y)\le 2K_1$, for all $x,y\in \ms$ due to \ref{A1}, it follows from Lemma \ref{bousquet} that inequality
\[\phi_n(\delta)\le \esp \phi_n(\delta)+\sqrt{\frac{2t}{n}\left(\sigma^2(\delta)+8K_1\,\esp \phi_n(\delta)\right)}+\frac{4K_1 t}{3n},\]
holds with probability at least $1-e^{-t}$.
Using basic inequalities $\sqrt{u+v}\le \sqrt{u}+\sqrt{v}$ and $2\sqrt{uv}\le u+v$ for positive numbers, we further deduce that 
\begin{equation}
\label{proof:ub:e1}
\phi_n(\delta)\le 2\esp \phi_n(\delta)+\sigma(\delta)\sqrt{\frac{2t}{n}}+\frac{16K_1 t}{3n},
\end{equation}
with probability at least $1-e^{-t}$.
Combining \ref{A2} and \ref{A3}, we deduce that for all $x\in\ms$,
\begin{equation}
\label{thm:barycentreemp:e2}
P(F(x,.)-F(x^*,.))^2\le K^2_2K^{\alpha}_3 (P(F(x,.)-F(x^*,.)))^{\alpha\beta},
\end{equation}
and therefore, 
\begin{equation}
\label{proof:ub:e2}
\sigma^2(\delta)\le K^2_2K^{\alpha}_3\delta^{\alpha\beta}.
\end{equation}

Next, we provide an upper bound for $\esp \phi_n(\delta)$.
Let $\sigma_1,\dots,\sigma_n$ be a sequence of i.i.d. random signs, \emph{i.e.} such that $\prob(\sigma_i=-1)=\prob(\sigma_i=1)=1/2$, independent from the $Y_i$'s.
The symmetrization principle (see, e.g., Lemma 7.4 in \cite{vanHa16}) indicates that
\begin{align}
\esp \phi_n(\delta)&\le 2 \esp\sup\left\{\frac{1}{n}\sum_{i=1}^{n}\sigma_i(F(x,Y_i)- F(x^*,Y_i)): x\in \ms(\delta)\right\}
\nonumber\\
&= 2 \esp\sup\left\{\frac{1}{n}\sum_{i=1}^{n}\sigma_i F(x,Y_i): x\in \ms(\delta)\right\},
\nonumber
\end{align}
where the last line follows from the fact that the $\sigma_i$'s are centered and independent of the $Y_i$'s.
Now, introduce the set $\mathcal F=\{F(x,.):x\in M\}$, and observe that, conditionally on the $Y_i$'s, the process 
\[X_f:=\frac{1}{\sqrt n}\sum_{i=1}^{n}\sigma_i f(Y_i),\quad f\in \mathcal F,\]
satisfies the separability condition of Lemma \ref{lem:chain} due to the separability of $M$.
In addition, this process is subgaussian since
\[\forall f,g\in \mathcal F,\forall\theta\in \R,\quad \log\esp[e^{\theta(X_f-X_g)}|Y_1,\dots,Y_n]\le\frac{\theta^2d_n(f,g)^2}{2},\]
where
\[d_n(f,g)^2=\frac{1}{n}\sum_{i=1}^n(f(Y_i)-g(Y_i))^2\]
is the natural metric in $L^2(P_n)$.
Also, it satisfies the Lipchitz condition \[|X_f-X_g|\le \sqrt{n}d_n(f,g).\] 
Hence, denoting $\mathcal F(\delta)=\{F(x,.):x\in \ms(\delta)\}$ and applying Lemma \ref{lem:chain}, we obtain 
\[\esp \phi_n(\delta)\le 2\esp\inf_{\e\ge 0}\left\{2\e+\frac{12}{\sqrt{n}} \int_{\e}^{+\infty}\sqrt{\log N( \mathcal F(\delta),d_n,u)}\,\textrm{d}u\right\}.\]
But combining \ref{A2} and \ref{A3}, we deduce that, almost surely,
\begin{align}
    N( \mathcal F(\delta),d_n,u)&\le N\left(M(\delta),d,\left(\frac{u}{K_2}\right)^{\frac{1}{\alpha}}\right)
    \nonumber\\
    &\le N\left(B(x^*,\sqrt{K_3\delta^{\beta}}),d,\left(\frac{u}{K_2}\right)^{\frac{1}{\alpha}}\right).
    \nonumber
\end{align}
As a result,
\begin{align}
    \esp \phi_n(\delta)&\le 2\inf_{\e\ge 0}\left\{2\e+\frac{12}{\sqrt{n}} \int_{\e}^{+\infty}\sqrt{\log N\left(B(x^*,\sqrt{K_3\delta^{\beta}}),d,\left(\frac{u}{K_2}\right)^{\frac{1}{\alpha}}\right)}\,\textrm{d}u\right\}
    \nonumber\\
    &= 2\inf_{\e\ge 0}\left\{2\e+\frac{12}{\sqrt{n}} \int_{\e}^{K_2K^{\alpha/2}_3\delta^{\alpha\beta/2}}\sqrt{\log N\left(B(x^*,\sqrt{K_3\delta^{\beta}}),d,\left(\frac{u}{K_2}\right)^{\frac{1}{\alpha}}\right)}\,\textrm{d}u\right\}
    \nonumber\\
    &\le 2\inf_{\e\ge 0}\left\{2\e+\frac{12}{\sqrt{n}} \int_{\e}^{K_2K^{\alpha/2}_3\delta^{\alpha\beta/2}}\sqrt{\frac{D}{\alpha}\log\left(\frac{C^{\alpha}K_2K^{\alpha/2}_3\delta^{\alpha\beta/2}}{u}\right)}\,\textrm{d}u\right\},
    \nonumber
\end{align}
where the last inequality follows from \ref{B1}.
Assuming without loss of generality that $C\ge 1$ and using the simple upper bound $\log(x)\le x-1\le x$, for all $x>0$, it follows from straightforward computations that
\begin{equation}
    \label{proof:ub:e3}
    \esp \phi_n(\delta)\le \frac{48C^{\frac{\alpha}{2}}D^{\frac{1}{2}}K_2K^{\frac{\alpha}{2}}_3}{\alpha^{\frac{1}{2}}}\cdot\delta^{\frac{\alpha\beta}{2}}n^{-\frac{1}{2}}.
\end{equation}

Combining \eqref{proof:ub:e1},\eqref{proof:ub:e2} and \eqref{proof:ub:e3} implies therefore that we have
\[\phi_n(\delta)\le b_n(\delta,t):= c_1\delta^{\frac{\alpha\beta}{2}}\sqrt{\frac{D}{n}}+c_2\delta^{\frac{\alpha\beta}{2}}\sqrt{\frac{t}{n}}+\frac{c_3 t}{n},\]
with probability at least $1-e^{-t}$ where 
\begin{equation}
\label{eq:csts}
c_1=\frac{96C^{\frac{\alpha}{2}}K_2K^{\frac{\alpha}{2}}_3}{\alpha^{\frac{1}{2}}},\quad c_2=\sqrt{2}K_2K^{\frac{\alpha}{2}}_3\quad\mbox{and}\quad c_3=\frac{16K_1}{3}.
\end{equation}
Using the fact that $\delta_n\le\phi_n(\delta_n)$ and Lemma \ref{lem:fixedpoint}, it follows that
\[\delta_n\le b_n(t):=\inf\left\{\tau>0:\sup_{\delta\ge \tau}\delta^{-1}b_n\left(\delta,\tfrac{t\delta}{\tau}\right)\le 1\right\},\]
with probability larger that $1-2e^{-t}$.
It therefore remains to provide an upper bound for $b_n(t)$.
Observing that $\alpha\beta\le 1$, and that for nonincreasing functions $h_j:[0,+\infty)\to [0,+\infty)$ we have
\[\inf\{\tau>0: h_1(\tau)+\dots+h_m(\tau)\le 1\}\le \max_{1\le j\le m}\inf\{\tau>0: h_j(\tau)\le 1/m\},\]
it follows that
\[b_n(t)\le \max\left\{(3c_1)^{\frac{2}{2-\alpha\beta}}\left(\frac{D}{ n}\right)^{\frac{1}{2-\alpha\beta}},(3c_2)^{\frac{2}{2-\alpha\beta}}\left(\frac{t}{n}\right)^{\frac{1}{2-\alpha\beta}},\frac{3c_3t}{n}\right\},\]
where $c_1,c_2,c_3$ are as in \eqref{eq:csts}.
To sum up, we have shown that, for all $n\ge 1$ and all $t>0$,
\begin{align*}
    &\int_M(F(x_n,y)-F(x^*,y))\,{\rm d}P(y)\\
    &\le \max\left\{(3c_1)^{\frac{2}{2-\alpha\beta}}\left(\frac{D}{ n}\right)^{\frac{1}{2-\alpha\beta}},(3c_2)^{\frac{2}{2-\alpha\beta}}\left(\frac{t}{n}\right)^{\frac{1}{2-\alpha\beta}},\frac{3c_3t}{n}\right\},
    \end{align*}
with probability at least $1-2e^{-t}$ where $c_1,c_2,c_3$ are as in \eqref{eq:csts}.
Finally note that, at the price of a slightly worst dependence on the constants, the last term in the maximum can be removed.
Indeed, if $t<n$, we have \[\frac{t}{n}\le\left(\frac{t}{n}\right)^{\frac{1}{2-\alpha\beta}},\] 
while, for $t\ge n$, assumption \ref{A3} implies that inequality 
\[
 \int_M(F(x_n,y)-F(x^*,y))\,{\rm d}P(y)\le 2K_3\left(\frac{t}{n}\right)^{\frac{1}{2-\alpha\beta}}\]
 trivially holds.
This completes the proof.
\end{proof}

\subsection{Proof of Example \ref{exm:finited}}
We prove inequalities \eqref{exm:finite3}.
The proof uses a classical strategy but is reproduced for completeness.
Define $\alpha_D:\ms\times(0,+\infty)\to [0,+\infty)$ by
 \begin{equation}
 \label{exm:finite1}
 \alpha_D(x,r)=\frac{\mu(B(x,r))}{r^D},
 \end{equation}
so that $\alpha_-\le\alpha_D(x,r)\le\alpha_+$ by assumption.
Let $x\in\ms$ and $0<\e\le r$ be fixed.
Let $x_1,\dots,x_N$ be a minimal $\e$-net for $B(x,r)$ of size $N=N(B(x,r),d,\e)$.
Then, by monotonicity and subadditivity of $\mu$, it follows that 
\[\alpha_D(x,r)r^D=\mu(B(x,r))\le\sum_{i=1}^N\mu(B(x_i,\e))=\e^D\sum_{i=1}^N\alpha_D(x_i,\e).\]
By definition of $\alpha_-$ and $\alpha_+$, we deduce that $\alpha_- r^D\le \alpha_+ N\e^{D}$ which proves the first inequality.
To prove the second inequality, define the $\e$-packing number $N_{\mathrm{pack}}(B(x,r),d,\e)$ of $B(x,r)$ as the maximal number $m$ of points $x_1,\dots,x_m\in B(x,r)$ such that $d(x_i,x_j)>\e$ for all $i\ne j$.
A collection of such points is called an $\e$-packing of $B(x,r)$.
It is a classical fact that the covering and packing numbers satisfy the duality property
\[N(B(x,r),d,\e)\le N_{\mathrm{pack}}(B(x,r),d,\e)\le N(B(x,r),d,\e/2).\]
In particular, to upper bound the $\e$-covering number of $B(x,r)$ it suffices to upper bound its $\e$-packing number.
Hence, let $x_1,\dots, x_m$ be a maximal $\e$-packing of $B(x,r)$ of size $m=N_{\mathrm{pack}}(B(x,r),d,\e)$.
Notice that the balls $B(x_i,\e/2)$, $i=1,\dots,m$ are disjoint by definition and included in $B(x,r+\e/2)$.
Hence, by monotonicity and additivity of $\mu$, it follows that 
\begin{align*}
\alpha_D\left(x,r+\frac{\e}{2}\right)\left(r+\frac{\e}{2}\right)^D=\mu\left(B\left(x,r+\frac{\e}{2}\right)\right)&\ge\mu\left(\bigcup_{i=1}^{m}B\left(x_i,\frac{\e}{2}\right)\right)\\
& =\sum_{i=1}^m\mu\left(B\left(x_i,\frac{\e}{2}\right)\right)\\
& =\left(\frac{\e}{2}\right)^D\sum_{i=1}^m\alpha_D\left(x_i,\frac{\e}{2}\right).
\end{align*}
The definition of $\alpha_-$ and $\alpha_+$ implies once again that
\[m\alpha_-\left(\frac{\e}{2}\right)^D\le \alpha_+\left(r+\frac{\e}{2}\right)^D,\]
and therefore
\[m\le \frac{\alpha_+}{\alpha_-}\left(\frac{2r}{\e}+1\right)^D\le\frac{\alpha_+}{\alpha_-}\left(\frac{3r}{\e}\right)^D,\]
which concludes the proof.

\subsection{Proof of Theorem \ref{thm:upperbound2}}
The proof is identical to that of Theorem \ref{thm:upperbound1} up to a minor modification.
Here, the control on the complexity of set $\ms$ is only global.
Hence, the inequality 
\[\esp \phi_n(\delta)\le 2\inf_{\e\ge 0}\left\{2\e+\frac{12}{\sqrt{n}} \int_{\e}^{K_2K^{\alpha/2}_3\delta^{\alpha\beta/2}}\sqrt{\log N\left(B(x^*,\sqrt{K_3\delta^{\beta}}),d,\left(\frac{u}{K_2}\right)^{\frac{1}{\alpha}}\right)}\,\textrm{d}u\right\},\]
used in the proof of Theorem \ref{thm:upperbound1}, while still valid, cannot be exploited as such.
We simply replace it by the upper bound
\[\esp \phi_n(\delta)\le 2\inf_{\e\ge 0}\left\{2\e+\frac{12}{\sqrt{n}} \int_{\e}^{K_2K^{\alpha/2}_3\delta^{\alpha\beta/2}}\sqrt{\log N\left(M,d,\left(\frac{u}{K_2}\right)^{\frac{1}{\alpha}}\right)}\,\textrm{d}u\right\}.\]
From then on, the proof is similar to that of Theorem \ref{thm:upperbound1}.

\subsection{Proof of Theorem \ref{thm:varequal}}

The proof of Theorem \ref{thm:varequal} follows by combining Lemmas \ref{lem:varequal} and \ref{lem:baryisexpbary} below.

\begin{defi}[Exponential barycenter]
\label{def:expbary}
For any $P\in\mathcal{P}_2(\ms)$, a point $x^{*}\in \ms$ is said to be an exponential barycenter of $P$ if 
\[
\int_\ms\int_\ms \langle \log_{x^{*}}(x),\log_{x^{*}}(y)\rangle_{x^{*}}\, \emph{d}P(x)\emph{d}P(y)=0.
\]
\end{defi}
\noindent For the definition of $\log_x$ and $\langle .,.\rangle_x$ we refer the reader to Appendix \ref{app:mg}.
Exponential barycenters where introduced in \cite{emery1991barycentre}.
The definition of an exponential barycenter mimics that of the Pettis integral of a Hilbert valued function and stands as an alternative way to define the analog of the mean value of an element of $\mathcal P_2(\ms)$.
Next is an important property of exponential barycenters.

\begin{thm}[Theorem 45 in \cite{yokota2012rigidity}]
\label{thm:yokota}
Suppose that ${\rm curv}(\ms)\ge 0$.
Let  $x^{*}$ be an exponential barycenter of $P\in\mathcal{P}_2(\ms)$.
Then the linear hull of the support of $P\circ \log^{-1}_{x^{*}}$ is isometric to a Hilbert space.
\end{thm}
\begin{rem}
Note that Theorem 45 of \cite{yokota2012rigidity} requires the tangent cone to be separable.
It is not clear to us whether the tangent cone of a separable geodesic space with curvature bounded below is always separable.
However, separability is used only to approximate (a countable number) of integrals w.r.t a measure by the same integral w.r.t. a finitely supported measure.
This can be derived from the law of large number and thus separability is not necessary.
\end{rem}
Using Theorem \ref{thm:yokota} we prove next that the identity of Theorem \ref{thm:varequal} holds for exponential barycenters.
\begin{lem}
\label{lem:varequal}
Suppose that ${\rm curv}(\ms)\ge 0$.
Let $P\in\mathcal P_2(\ms)$ and $x^*$ be an exponential barycenter of $P$.
Then, for all $x\in \ms$,
\[d(x,x^*)^2 \int_{\ms}k^x_{x^*}(y)\,{\rm d}P(y)= \int_{\ms}(d(x,y)^2-d(x^*,y)^2)\,{\rm d}P(y),\]
where, for all $x\ne x^*$ and all $y$, $k^x_{x^*}(y)$ is as in \eqref{eq:k}.
\end{lem}
\begin{proof}[Proof of Lemma \ref{lem:varequal}]
Fix an exponential barycenter $x^{*}$ of $P$.
For brevity, denote $\log=\log_{x^*}$, $\|.\|=\|.\|_{x^*}$ and $\langle .,.\rangle=\langle .,.\rangle_{x^*}$.
Given any $x\in \ms$, let $x_{t}$ be a geodesic connecting $x^{*}$ to $x$ in $\ms$.
It is an easy exercise to check that $t\mapsto\log(x_t)$ is a geodesic connecting $\log(x^*)$ to $\log(x)$ in $T_{x^{*}}\ms$.
According to Theorem \ref{thm:yokota}, and the geometry of a Hilbert space, it follows that, for all $x,y\in \ms$ and all $t\in [0,1]$,
\[
\|y-x_t\|^2 =(1-t)\|y-x^{*}\|^2+t\|y-x\|^2-t(1-t)\|x-x^*\|^2,
\]
where here and throughout, we denote both a point $x\in \ms$ and its image $\log(x)$ in the tangent cone $T_{x^{*}}\ms$ by the same symbol $x$ when there is no risk of confusion.
Now since $\mathrm{curv}(\ms)\ge 0$, we have $d(x,y)\le \|x-y\|$ for all $x,y\in \ms$, with equality if $x=x^{*}$ or $y=x^{*}$.
Hence, it follows from the previous identity that, for all $x,y\in \ms$ and all $t\in (0,1)$,
\begin{align*}
t(1-t)d(x,x^*)^2&= (1-t)d(x^*,y)^2+t\|y-x\|^2-\|y-x_t\|^2\\
&=t( d(x,y)^2 - d(x^*,y)^2)+ (\|y-x^*\|^2 - \|y-x_t\|^2)\\
&\quad+t\left(1-k^x_{x^*}(y)\right)d(x,x^*)^2.
\end{align*}
Reordering terms and dividing by $t$, we obtain
\begin{equation}
\label{eq:propkx}
(k^x_{x^*}(y)-t)d(x,x^*)^2=(d(x,y)^2 - d(x^*,y)^2)+ \frac{1}{t}(\|y-x^*\|^2 - \|y-x_t\|^2).
\end{equation}
Integrating with respect to $P({\textrm d}y)$, we obtain
\[
(P k^x_{x^*}(.)-t)d(x,x^*)^2= P(d(x,.)^2 - d(x^*,.)^2)+ \frac{P(\|x^{*}-.\|^2 - \|x_t-.\|^2)}{t}.
\]
Finally, using Theorem \ref{thm:yokota}, we get that
\[P(\|x^{*}-.\|^2 - \|x_t-.\|^2)=-t^2\|x^*-x\|^2.\]
Hence, letting $t\to 0$ in the previous identity leads to the desired result.
\end{proof}

\begin{lem}
\label{lem:baryisexpbary}
Suppose that ${\rm curv}(\ms)\ge 0$ and let $P\in\mathcal{P}_2(\ms)$.
Then a barycenter of $P$ is an exponential barycenter of $P$.
\end{lem}

\begin{proof}[Proof of Lemma \ref{lem:baryisexpbary}]
Fix a barycenter $x^*$ of $P$.
As in the proof of the previous Lemma, denote $\log=\log_{x^*}$, $\|.\|=\|.\|_{x^*}$ and $\langle .,.\rangle=\langle .,.\rangle_{x^*}$ for brevity.
Also, we denote both a point $x\in \ms$ and its image $\log(x)$ in the tangent cone $T_{x^{*}}\ms$ by the same symbol $x$ when there is no risk of confusion.
Since $\mathrm{curv}(\ms)\ge 0$, the first statement of Theorem 45 in \cite{yokota2012rigidity} implies that 
\[
\int_\ms\int_\ms \langle x,y\rangle\, \textrm{d}P(x)\textrm{d}P(y)\ge 0.
\]
Moreover, we know that $d(x,y)\le \|x-y\|$ for all $x,y\in \ms$ with equality if $x=x^{*}$ or $y=x^{*}$.
In particular, $x^{*}$ also minimises 
\[
y\mapsto \int_\ms \|x-y\|^2\,\textrm{d}P(x).
\]
Thus, for all $y\in \ms$, letting $y_t$ be a geodesic connecting $x^{*}$ to $y$ we get for all $t\in(0,1]$,
\begin{align*}
\int_\ms \|x\|^2\,\textrm{d}P(x)&=\int_\ms \|x-x^*\|^2\,\textrm{d}P(x)\\
&\le\int_\ms \|x-y_t\|^2\,\textrm{d}P(x)\\
&=\int_\ms \|x-t\cdot y\|^2\,\textrm{d}P(x)\\
&=\int_\ms (\|x\|^2 - 2t\langle x, y\rangle + t^2\|y\|^2)\,\textrm{d}P(x),
\end{align*}
where we have used the properties $\log$ and $\|.\|$ stated in section \ref{subsec:cone}.
Simplifying the above expression, we obtain, for all $t\in(0,1]$,
\[
2\int_\ms \left\langle x, y\right\rangle\,\textrm{d}P(x)\le t\|y\|^2.
\]
Letting $t\rightarrow 0$, we get
\[
\int_\ms \langle x,y\rangle\textrm{d}P(x)\le 0.
\]
Integrating with respect to $y$, we obtain 
\[
\int_\ms\int_\ms \langle x,y\rangle\, \textrm{d}P(x)\textrm{d}P(y)\le 0.
\]
Combining this observation with the first inequality of the proof shows that $x^*$ is an exponential barycenter.
\end{proof}

\subsection{Proof of Theorem \ref{thm:extendgeod}}
Consider $y$ in the support of $P$, denote $y_{\lambda}=\gamma^+_y(1+\lambda)=e_{\lambda}(y)$ and consider the map $\sigma_y:[0,1]\to\ms$ defined by $\sigma_y(t)=\gamma^+_y(t(1+\lambda))$.
By assumption, $\sigma_{y}$ is a geodesic connecting $x^*$ to $y_{\lambda}$.
In addition, we have by construction that \[\sigma_{y}(\tau)= y\quad\mbox{where}\quad\tau=\frac{1}{1+\lambda}.\]
It follows from the properties of the map $\log_y$ listed in appendix \ref{app:mg} that
\[y=(1-\tau)x^*+\tau y_{\lambda},\]
where we identify a point $u$ and its image $\log_y(u)$ in $T_y\ms$.
Now since $\mathrm{curv}(\ms)\ge 0$ (see Proposition \ref{pro:nongeoboundedcurv}), we know that $\mathrm{curv}(T_y\ms)\ge 0$ so that, for all $x\in T_y\ms$,
\begin{align*}
    \|x-y\|^2_y&\ge(1-\tau)\|x-x^*\|^2_y+\tau\|x-y_{\lambda}\|^2_y-\tau(1-\tau)\|y_{\lambda}-x^*\|^2_y\\
    &=\frac{\lambda}{1+\lambda}\|x-x^*\|^2_y+\frac{1}{1+\lambda}\|x-y_{\lambda}\|^2_y-\frac{\lambda}{(1+\lambda)^2}\|y_{\lambda}-x^*\|^2_y.
\end{align*}
Using the fact that $\|y_{\lambda}-x^*\|^2_y=(1+\lambda)^2\|y-x^*\|^2_y$ and the fact that $d(u,v)\le \|u-v\|_y$, with equality if $u=y$ or $v=y$, we deduce from the inequality above that
\begin{align*}
    \frac{\lambda}{1+\lambda}d(x,x^*)^2&\le\frac{\lambda}{1+\lambda}\|x-x^*\|^2_y\\
    &\le d(x,y)^2-\frac{1}{1+\lambda}\|x-y_{\lambda}\|^2_y+\lambda d(x^*,y)^2\\
    &=(d(x,y)^2-d(x^*,y)^2)-\frac{1}{1+\lambda}\|x-y_{\lambda}\|^2_y+(1+\lambda)d(x^*,y)^2.
\end{align*}
Integrating this inequality with respect to ${\rm d}P(y)$, it follows that
\begin{align*}
    \frac{\lambda}{1+\lambda}d(x,x^*)^2&\le\int_{\ms} (d(x,y)^2-d(x^*,y)^2){\rm d}P(y)\\
    &+\int_{\ms}\left((1+\lambda)d(x^*,y)^2-\frac{1}{1+\lambda}\|x-y_{\lambda}\|^2_y\right){\rm d}P(y).
\end{align*}
To conclude the proof, it remains to show that, for all $x\in\ms$, 
\[\rho(x):=\int_{\ms}\left((1+\lambda)d(x^*,y)^2-\frac{1}{1+\lambda}\|x-y_{\lambda}\|^2_y\right){\rm d}P(y)\le0.\]
Observing that  \[(1+\lambda)^2d(x^*,y)^2=d(x^*,y_{\lambda})^2\quad\mbox{and that}\quad d(x,y_{\lambda})\le\|x-y_{\lambda}\|_{y},\]
we deduce that
\begin{align*}
    \rho(x)&\le\frac{1}{1+\lambda}\int_{\ms}(d(x^*,y_{\lambda})^2-d(x,y_{\lambda})^2){\rm d}P(y)\\
    &=\frac{1}{1+\lambda}\int_{\ms}(d(x^*,y)^2-d(x,y)^2){\rm d}P_{\lambda}(y),
\end{align*}
where we have used the fact that $y_{\lambda}=e_{\lambda}(y)$ and the fact that $P_{\lambda}=(e_{\lambda})_*P$.
Hence, for all $x\in \ms$, inequality $\rho(x)\le 0$ follows from the fact that $x^*$ is a barycenter of $P_{\lambda}$ by assumption.

\subsection{Proof of Theorem \ref{thm:extendgeodp2}}

We start by a technical lemma.
Below $H$ is a Hilbert space with scalar product $\langle.,.\rangle$ and associated norm $\|.\|$.
\begin{lem}\label{lem:phiconvsub}
Let $\phi:H\to \R$ be a convex function and $\partial\phi\subset H^2$ its subdifferential defined by
\[
(x,y)\in\partial\phi \Leftrightarrow \forall z\in H,\quad\phi(z)\ge\phi(x)+\langle y,z-x\rangle.
\]
Then, for all $c>0$,
\[
(x,y)\in\partial\phi\Leftrightarrow\forall z\in H,\quad\phi(z)\ge\phi(x)+\langle y,z-x\rangle -c\|x-z\|^2.
\]
\end{lem}

\begin{proof}[Proof of the Lemma]
One implication is obvious.
For the second implication, suppose $(x,y)\in H^2$ is such that
\begin{equation}
\label{lem:p2e1}
\forall z\in H,\quad\phi(z)\ge\phi(x)+\langle y,z-x\rangle -c\|z-x\|^2.
\end{equation}
Then, on the one hand, we get by convexity of $\varphi$ that, for all $z\in H$ and all $t\in(0,1)$, 
\begin{equation}
\label{lem:p2e2}
\phi(z)-\phi(x)\ge \frac{1}{t}(\phi(tz+(1-t)x)-\phi(x)).
\end{equation}
On the other hand, applying \eqref{lem:p2e1}, we obtain for all $z\in H$ and all $t\in(0,1)$,
\begin{equation}
\label{lem:p2e3}
\phi(tz+(1-t)x)-\phi(x) \ge t\langle y,z-x\rangle - ct^2\|z-x\|^2.
\end{equation}
Hence, combining \eqref{lem:p2e2} and \eqref{lem:p2e3}, we deduce that for all $z\in H$ and all $t\in(0,1)$,
\[\phi(z)-\phi(x)\ge \langle y,z-x\rangle - ct\|z-x\|^2.\]
Letting $t\to 0$ proves the other implication.
\end{proof}

We are now in position to prove Theorem \ref{thm:extendgeodp2}.
Fix $\mu,\nu\in S=\mathcal P_2(H)$ and denote $\gamma:[0,1]\to S$ a constant speed shortest path between $\mu$ and $\nu$.
By the Knott-Smith optimality criterion (see Theorem 2.12 in \cite{villani2003}), $\pi$ is an optimal transport plan of $(\mu,\nu)$ if and only if its support lies in the graph of the subdifferential of a convex function $\phi$, i.e.
\begin{equation}\label{eq:phisub}
(x,y)\in\supp \pi \Rightarrow \forall z\in H,\quad \phi(z)\ge \phi(x)+\langle y,z-x\rangle.
\end{equation}

Suppose first that, for some $\lambda >0$, $\gamma:[0,1]\to S$ can be extended to a function $\gamma^+:[0,1+\lambda]\to S$ that remains a shortest path between its endpoints $\mu=\gamma^+(0)=\gamma(0)$ and $\nu^{\lambda}:=\gamma^+(1+\lambda)$.
Then, by Theorem 7.2.2 in \cite{ambrosio2008gradient}, there exists an optimal transport plan $\pi^\lambda$ of $(\mu,\nu^\lambda)$ such that $\{\pi\}=\Gamma_o(\mu,\nu)$ is that the law of 
\[
(X,Y):=\left(X,\frac{\lambda}{1+\lambda}X+\frac{1}{1+\lambda}Y^\lambda\right)
\]
where $(X,Y^\lambda)\sim \pi^\lambda$.
In particular $Y^\lambda=(1+\lambda Y)-\lambda X$.
Therefore, there exists a convex function $\phi^\lambda$ such that denoting $y^\lambda=(1+\lambda)y-\lambda x$
\begin{align}
(x,y)\in\supp\pi&\Leftrightarrow (x,y^\lambda)\in\supp\pi^\lambda\nonumber\\
&\Rightarrow\forall z\in H, \phi^\lambda(z)\ge \phi^\lambda(x)+\langle y^\lambda,z-x\rangle\nonumber\\
&\Leftrightarrow \forall z\in H, \phi^\lambda(z)\ge \phi^\lambda(x)+(1+\lambda)\langle y,z-x\rangle-\lambda \langle x,z-x\rangle\nonumber\\
&\Leftrightarrow \forall z\in H, \phi^\lambda(z)+\frac{\lambda}{2}\|z\|^2\ge \phi^\lambda(x)+\frac{\lambda}{2}\|x\|^2+(1+\lambda)\langle y,z-x\rangle+\lambda \frac{1}{2}\|x-z\|^2\nonumber\\
&\Rightarrow \forall z\in H, \frac{\phi^\lambda(z)}{1+\lambda}+\frac{\lambda\|z\|^2}{2(1+\lambda)}\ge \frac{\phi^\lambda(x)}{1+\lambda}+\frac{\lambda\|x\|^2}{2(1+\lambda)}+\langle y,z-x\rangle\nonumber\\
&\Leftrightarrow \forall z\in H, \phi(z)\ge \phi(x)+\langle y,z-x\rangle,\label{eq:phil}
\end{align}
where we denote 
\[
\phi=\frac{\phi^\lambda}{1+\lambda}+\frac{1}{2}\frac{\lambda}{1+\lambda}\|.\|^2.
\]
Thus, $\supp \pi$ lies in the subdifferential of $\phi$ that is $\frac{\lambda}{1+\lambda}$-convex, since $\phi^\lambda$ is convex.

Conversely, suppose that there exists a $\frac{\lambda}{1+\lambda}$-convex function $\phi$ such that $\supp\pi$ for $\pi\in\Gamma_o(\mu,\nu)$ lies in the subdifferential of $\phi$.
Denote $(X,Y)\sim\pi$ and set $Y^\lambda=Y+\lambda(Y-X)\sim\nu^\lambda$ and $(X,Y^\lambda)\sim\pi^\lambda$.
Then $\pi^\lambda$ is an optimal transport plan between $\mu$ and $\nu^\lambda$ if and only if there exists a convex function $\phi^\lambda$ such that the $\supp\pi^\lambda$ lies in $\partial\phi^\lambda$.
In that case, 
\[
W_2^2(\mu,\nu^\lambda)=E\|Y^\lambda-X\|^2=(1+\lambda)^2W_2^2(\mu,\nu)=\frac{(1+\lambda)^2}{\lambda^2}W_2^2(\nu,\nu^\lambda),
\]
so that by Lemma 7.2.1 of \cite{ambrosio2008gradient}, $\nu$ is in the shortest path joining $\mu$ to $\nu^\lambda$, which is the desired result.
It thus just remains to prove that there exists a convex function $\phi^\lambda$ such that $\supp\pi^\lambda$ lies in $\partial\phi^\lambda$.

Set 
\[
\phi^\lambda=(1+\lambda)\phi - \frac{1}{2}\lambda\|.\|^2.
\]
$\phi^\lambda$ is convex since $\phi$ is $\frac{\lambda}{1+\lambda}$-convex.
Then, denoting again $y^\lambda=(1+\lambda)y-\lambda x$,
\begin{align}
(x,y^\lambda)\in\supp\pi^\lambda&\Leftrightarrow(x,y)\in\supp\pi\nonumber\\
&\Rightarrow\forall z\in H, \phi(z)\ge \phi(x)+\langle y,z-x\rangle\nonumber\\
&\Leftrightarrow \forall z\in H, \frac{\phi^\lambda(z)}{1+\lambda}+\frac{\lambda\|z\|^2}{2(1+\lambda)}\ge \frac{\phi^\lambda(x)}{1+\lambda}+\frac{\lambda\|x\|^2}{2(1+\lambda)}+\langle y,z-x\rangle\nonumber\\
&\Leftrightarrow\forall z\in H, \phi^\lambda(z)\ge \phi^\lambda(x)+\langle y^\lambda,z-x\rangle-\lambda\frac{1}{2}\|x-z\|^2.\label{eq:phil2}
\end{align}
Since $\phi^\lambda$ is convex, by Lemma \ref{lem:phiconvsub}, \eqref{eq:phil2} is equivalent to 
\[
\forall z\in H, \phi^\lambda(z)\ge \phi^\lambda(x)+\langle y^\lambda,z-x\rangle,
\]
that is, $\supp \pi^\lambda$ lies in the subdifferential of the convex function $\phi^\lambda$.

\subsection{Proof of Theorem \ref{thm:fvi}}
For all $x\in M$, denote
\[V(x)=\int_{\ms} F(x,y)\,\mathrm{d}P(y).\]
Now fix $x\in\ms$ and let $\gamma:[0,1]\to M$ be a path connecting $x^*$ to $x$ and along which $V$ is $(k,\beta)$-convex.
Then, for all $t\in[0,1]$,
\[V(\gamma_t)\le (1-t)V(x^*)+tV(x)-kt(1-t)d(x,x^*)^{\frac{2}{\beta}}.\]
Reordering these terms we obtain, for all $t\in(0,1)$,  
\begin{align*}
kd(x,x^*)^{\frac{2}{\beta}}&\le \frac{V(x)-V(x^*)}{1-t}+\frac{V(x^*)-V(\gamma_t)}{t(1-t)}\\
&\le \frac{V(x)-V(x^*)}{1-t},
\end{align*}
since $V(x^*)\le V(\gamma_t)$ by definition of $x^*$.
Letting $t$ tend to $0$ concludes the proof.

\subsection{Proof of Theorem \ref{thm:dfA2}}

For any $\mu\in M\subset\mathcal P_2(E)$, let $g_{\mu}$ be the density of $\mu$ with respect to the reference measure $m$.
Fix $\nu,\mu,\mu'\in M$ and denote for brevity 
\[a_\mu(x)=\frac{g_\mu}{g_\nu}(x),\quad a_{\mu'}(x)=\frac{g_{\mu'}}{g_\nu}(x)\myand \delta_{\mu,\mu'}(x)=\dfrac{ f(a_{\mu}(x))- f(a_{\mu'}(x))}{ a_{\mu}(x)- a_{\mu'}(x)}.\]
Then, we obtain
\begin{align}
    D_f(\mu,\nu)-D_f(\mu',\nu) &=\int( f(a_{\mu}(x))- f(a_{\mu'}(x)))g_\nu(x)\mathrm{d}m(x)
    \nonumber\\
    &=\int \dfrac{ f(a_{\mu}(x))- f(a_{\mu'}(x))}{ a_{\mu}(x)- a_{\mu'}(x)}(\mathrm{d}\mu(x)-\mathrm{d}\mu'(x))
    \nonumber\\
    &=\int\delta_{\mu,\mu'}(x)(\mathrm{d}\mu(x)-\mathrm{d}\mu'(x)).
    \label{eq:df1}
\end{align}
Let us prove that $\delta_{\mu,\mu'}$ is Lipschitz.
To that aim, observe that since $f'$ is $L$-Lipschitz, for all $x,y\in E$, 
\begin{align}
|\delta_{\mu,\mu'}(x)-\delta_{\mu,\mu'}(y)|&=|\int_0^1 \left(f'((1-t)a_{\mu}(x)+ta_{\mu'}(x))-f'((1-t)a_{\mu}(y)+ta_{\mu'}(y))\right)\mathrm{d}t|\nonumber\\
&\le \int_0^1 |f'((1-t)a_{\mu}(x)+ta_{\mu'}(x))-f'((1-t)a_{\mu}(y)+ta_{\mu'}(y))|\mathrm{d}t\nonumber\\
&\le L\int_0^1 |(1-t)(a_{\mu}(x)-a_{\mu}(y))+t(a_{\mu'}(x)-a_{\mu'}(y))|\mathrm{d}t\nonumber\\
&\le L\max\{|a_{\mu}(x)-a_{\mu}(y)|,|a_{\mu'}(x)-a_{\mu'}(y)|\}.
\label{eq:df2}
\end{align}
Then, we see that
\begin{align*}
    a_{\mu}(x)-a_{\mu}(y)&=\frac{g_{\mu}(x)g_{\nu}(y)-g_{\mu}(y)g_{\nu}(x)}{g_{\nu}(x)g_{\nu}(y)}\\
    &=\frac{g_{\mu}(x)(g_{\nu}(y)-g_{\nu}(x))+g_{\nu}(x)(g_{\mu}(x)-g_{\mu}(y))}{g_{\nu}(x)g_{\nu}(y)},
\end{align*}
which implies, under the conditions of the theorem, that
\begin{equation}
\label{eq:df3}
|a_{\mu}(x)-a_{\mu}(y)|\le \frac{2\Lambda c_+}{c^2_-}\|x-y\|.
\end{equation}
Combining \eqref{eq:df2} and \eqref{eq:df3} we deduce that $\delta_{\mu,\mu'}$ is Lipschitz with constant at most $2L\Lambda c_+/c^2_-$.
Using \eqref{eq:df1} and the Kantorovich-Rubinstein formula (see remark 6.5 in \cite{villani2008optimal}) we therefore obtain that 
\[|D_f(\mu,\nu)-D_f(\mu',\nu)|\le \frac{2L\Lambda c_+}{c^2_-}W_1(\mu,\mu').\]
The result follows by observing that $W_1(\mu,\mu')\le W_2(\mu,\mu')$ (see remark 6.6 in \cite{villani2008optimal}).

\subsection{Proof of Theorem \ref{thm:dfA3}}
First note that the second statement of the theorem follows by Theorem \ref{thm:fvi}.
Hence, we need only to prove that, for all $\mu_0,\mu_1\in M$, there exists a geodesic $\mu_t$ connecting $\mu_0$ to $\mu_1$ such that, for all $0\le t\le 1$, 
\[D_f(\mu_t,\nu)\le (1-t)D_f(\mu_0,\nu)+tD_f(\mu_1,\nu)-c\lambda t(1-t)W_2(\mu_0,\mu_1)^2.\]

If either $\mu_0$ or $\mu_1$ is not absolutely continuous with respect to $\nu$, then the right hand side is $+\infty$ and the inequality trivially holds.
Suppose now that $\mu_0\ll \nu$ and $\mu_1\ll \nu$.
Since $\nu$, and therefore $\mu_0$, has a density with respect to the Lebesgue measure, there exists an optimal transport map $T:E\to E$ pushing $\mu_0$ to $\mu_1$.
Letting $T_t(x)=(1-t)x+tT(x)$, the curve $\mu_t=(T_t)_{\#}\mu_0$ defines a geodesic connecting $\mu_0$ to $\mu_1$ in $M$. For all $t\in[0,1]$, we denote 
\[\rho_t(x)=\frac{{\rm d}\mu_t}{{\rm d}x}(x),\]
so that, letting ${\rm d}\nu(x)=e^{-U(x)}{\rm d}x$, 
\[{\rm d}\mu_t(x)=\rho_t(x) e^{U(x)}{\rm d}\nu(x).\]
Letting $D T_t(x)=(1-t)I+tDT(x)$ denote the differential of $T_t$ at $x$, the change of variables formula 
\[\rho_0(x)=\rho_t(T_t(x)){\rm det}(DT_t(x)),\]
implies that
\begin{align}
    D_f(\mu_t,\nu)& =\int f(\rho_t(x)e^{U(x)})e^{-U(x)}\,{\rm d}x
    \nonumber\\
    & =\int f\left(\frac{\rho_0(x)e^{U({T}_t(x))}}{{\rm det}(DT_t(x))}\right)\frac{{\rm det}(DT_t(x))}{\rho_0(x)e^{U({T}_t(x))}}\rho_0(x)\,{\rm d}x
    \nonumber\\
    &=\int h(s(t,x))\rho_0(x)\,{\rm d}x,
    \nonumber
\end{align}
where 
\begin{equation}
\label{proof:thm:internal:e1}
s(t,x)=-U(T_t(x))+\log{\rm det}(DT_t(x))-\log\rho_0(x).
\end{equation}
The transport map $T$ being the gradient of a convex function, the map $t\in[0,1]\mapsto \log{\rm det}(DT_t(x))$ is concave for any fixed $x\in E$.
The $\lambda$-convexity of $U$ therefore implies that, for all $t\in [0,1]$,
\[s(t,x)\ge (1-t)s(0,x)+ts(1,x)+\lambda t(1-t)\|x-{\rm T}(x)\|^2.\]
Using the assumptions on $h$, which imply in particular that it is decreasing, we deduce that
\begin{align}
    h(s(t,x))&\le h((1-t)s(0,x)+ts(1,x)+\lambda t(1-t)\|x-{\rm T}(x)\|^2)
    \nonumber\\
    &\le h((1-t)s(0,x)+ts(1,x))-c\lambda t(1-t)\|x-{\rm T}(x)\|^2
    \nonumber\\
    &\le (1-t)h(s(0,x))+th(s(1,x))-c\lambda t(1-t)\|x-{\rm T}(x)\|^2.
    \nonumber
\end{align}
Integrating the last inequality with respect to $\rho_0(x){\rm d}x$ yields
\[D_f(\mu_t,\nu)\le (1-t)D_f(\mu_0,\nu)+tD_f(\mu_1,\nu)- c\lambda t(1-t)W^2_2(\mu_0,\mu_1),\]
which is the desired result.

\subsection{Proof of Theorem \ref{thm:dflinear}}
Denote ${\rm d}\mu_i=g_i{\rm d}\nu$.
The linear interpolation $\ell_t=(1-t)\mu_0+t\mu_1$ therefore satisfies
\[\frac{{\rm d}\ell_t}{{\rm d}\nu}=(1-t)g_0+tg_1.\]
It follows from the $k$-convexity of $f$ that
\begin{align*}
D_f(\ell_t,\nu) &=\int f\left(\frac{{\rm d}\ell_t}{{\rm d}\nu}\right)\,{\rm d}\nu\\
&=\int f((1-t)g_0+tg_1)\,{\rm d}\nu\\
&\le (1-t)\int f(g_0)\,{\rm d}\nu+t\int f(g_1)\,{\rm d}\nu-k t(1-t)\int |g_1-g_0|^2\,{\rm d}\nu\\
&= (1-t)D_f(\mu_0,\nu)+tD_f(\mu_1,\nu)-kt(1-t)\int |g_1-g_0|^2\,{\rm d}\nu.
\end{align*}
According to Theorem 6.15 in \cite{villani2008optimal}, we know that 
\begin{align*}
    W_2(\mu_0,\mu_1)^2 &\le 2\int \|x\|^2|g_1(x)-g_0(x)|\,{\rm d}\nu(x)\\
    &\le 2m_4(\nu)^{1/2}\left(\int |g_1(x)-g_0(x)|^2\,{\rm d}\nu(x)\right)^{1/2}.
    \end{align*}
Hence, we deduce that  
\[D_f(\ell_t,\nu)\le (1-t)D_f(\mu_0,\nu)+tD_f(\mu_1,\nu)-\frac{k}{4m_4(\nu)}t(1-t)W_2(\mu_0,\mu_1)^4,\]
which proves the first statement.
The second statement follows directly from Theorem \ref{thm:fvi}.

\subsection{Proof of Theorem \ref{thm:interactionA2}}
For fixed $y\in E$, the fact that $x\in E\mapsto g(x,y)$ is $L$-Lipschitz, the Kantorovich-Rubinstein formula (see remark 6.5 in \cite{villani2008optimal}) and the fact that $W_1\le W_2$ (see remark 6.6 in \cite{villani2008optimal}) implies that, for all $\mu,\mu'\in\mathcal P_2(E)$,
\[|\int g(x,y)(\mathrm{d}\mu(x)-\mathrm{d}\mu'(x))|\le LW_1(\mu,\mu')\le LW_2(\mu,\mu').\]
Integrating with respect to $\nu\in \mathcal P_2(E)$ implies that
\begin{align*}
    |I_g(\mu,\nu)-I_g(\mu',\nu)|&\le\int|\int g(x,y)(\mathrm{d}\mu(x)-\mathrm{d}\mu'(x))|\mathrm{d}\nu(y)\\
    &\le LW_2(\mu,\mu').
\end{align*}

\subsection{Proof of Theorem \ref{thm:interactionA3}}

Given that $(E,\rho)$ is separable, complete, locally compact and geodesic, the space $(\mathcal P_2(E),W_2)$ is also geodesic according to Corollary 7.22 in \cite{villani2008optimal}.
In addition, given $\mu_0,\mu_1\in\mathcal P_2(E)$, a curve $(\mu_t)_{0\le t\le 1}$ in $\mathcal P_2(E)$ is a geodesic connecting $\mu_0$ to $\mu_1$ if and only if there exists a probability measure $\Pi$ on the set $\mathcal G(E)$ of all geodesics $\gamma:[0,1]\to E$ in $E$ such that 
\[\mu_t=(\mathrm{eval}_t)_*(\Pi),\]
where $\mathrm{eval}_t(\gamma)=\gamma_t$, for all $\gamma\in\mathcal G(E)$, and where $(\mathrm{eval}_0,\mathrm{eval}_1)_*(\Pi)$ is an optimal coupling of $\mu_0$ and $\mu_1$.
For such a geodesic, 

\begin{align*}
I_g(\mu_t,\nu)&=\int\int g(x,y)\,{\rm d}\mu_t(x){\rm d}\nu(y)\\
    & =\int\int g(\gamma_t,y)\,{\rm d}\Pi(\gamma){\rm d}\nu(y)\\
    & \le\int\int \left((1-t)g(\gamma_0,y)+tg(\gamma_1,y)-kt(1-t)\rho(\gamma_0,\gamma_1)^\frac{2}{\beta}\right)\,{\rm d}\Pi(\gamma){\rm d}\nu(y)\\
    &= (1-t)I_g(\mu_0,\nu)+tI_g(\mu_1,\nu)-kt(1-t)\int\rho(\gamma_0,\gamma_1)^\frac{2}{\beta}\,{\rm d}\Pi(\gamma)\\
    &\le (1-t)I_g(\mu_0,\nu)+tI_g(\mu_1,\nu)-kt(1-t)\left(\int\rho(\gamma_0,\gamma_1)^2\,{\rm d}\Pi(\gamma)\right)^{\frac{1}{\beta}}\\
    &= (1-t)I_g(\mu_0,\nu)+tI_g(\mu_1,\nu)-kt(1-t)W_2(\mu_0,\mu_1)^{\frac{2}{\beta}}.
\end{align*}
This completes the proof of $(1)$.
Statement $(2)$ follows directly by combining the first statement and Theorem \ref{thm:fvi}.

\appendix

\section{Metric geometry}
\label{app:mg}
\subsection{Geodesic spaces}
\label{subsec:geodesic}
Let $(\ms,d)$ be a metric space.
We call path in $\ms$ a continuous map $\gamma:I\to \ms$ defined on an interval $I\subset \R$.
The length $L(\gamma)\in[0,+\infty]$ of a path $\gamma:I\to \ms$ is defined by 
\[
L(\gamma):=\sup\sum_{i=0}^{n-1}d(\gamma(t_i),\gamma(t_{i+1})),\]
where the supremum is taken over all $n\ge 1$ and all $t_0\le\dots\le t_{n}$ in $I$.
A path is called rectifiable if it has finite length.
Two paths $\gamma_1$ and $\gamma_2$ are said to be equivalent if $\gamma_1\circ\varphi_1=\gamma_2\circ\varphi_2$ for non-decreasing and continuous functions $\varphi_1$ and $\varphi_2$.
In this case, $\gamma_1$ is said to be a reparametrisation of $\gamma_2$ and we check that $L(\gamma_1)=L(\gamma_2)$. A path $\gamma:[a,b]\to \ms$ is said to have constant speed if for all $a\le s\le t\le b$,
\begin{equation}
\label{cspeed} L(\gamma_{[s,t]})=\frac{t-s}{b-a}L(\gamma),
\end{equation}
where $\gamma_{[s,t]}$ denotes the restriction of $\gamma$ to $[s,t]$.

\begin{pro}
Any rectifiable path has a constant speed reparametrisation $\gamma:[0,1]\to \ms$.
\end{pro}

Given $x,y \in \ms$, a path $\gamma:[a,b]\to \ms$ is said to connect $x$ to $y$ if $\gamma(a)=x$ and $\gamma(b)=y$.
By construction of the length function $L$, $d(x,y)\le L(\gamma)$ for any path $\gamma$ connecting $x$ to $y$.
The space $\ms$ is called a length space if, for all $x,y\in \ms$,
\begin{equation}
\label{dintrinsic}
d(x,y)=\inf_{\gamma} L(\gamma),
\end{equation}
where the infimum is taken over all paths $\gamma$ connecting $x$ to $y$.
A length space is said to be a geodesic space if, for all $x,y\in \ms$, the infimum on the right hand side of \eqref{dintrinsic} is attained.
\begin{defi}
In a geodesic space, we call geodesic between $x$ and $y$ any constant speed reparametrisation $\gamma:[0,1]\to \ms$ of a path attaining the infimum in \eqref{dintrinsic}.
\end{defi}
For a geodesic $\gamma$, it follows from its minimising properties that 
\[d(\gamma(s),\gamma(t))=L(\gamma_{[s,t]}),\]
for all $0\le s\le t\le 1$.
In particular, \eqref{cspeed} translates in this case as
\[d(\gamma(s),\gamma(t))=(t-s)d(\gamma(0),\gamma(1)),\]
for all $0\le s\le t\le 1$.
We end by a general characterization of geodesic spaces.
\begin{pro}
\label{def:midpoint}
Let $(M,d)$ be a metric space.
\begin{itemize}
\item[$(1)$] If $\ms$ is a geodesic space, then any two points $x,y\in \ms$ admit a midpoint, i.e. a point $z\in \ms$ such that
\[d(x,z)=d(y,z)=\frac{1}{2}d(x,y).\]
\item[$(2)$] Conversely, if $\ms$ is complete and if any two points in $M$ admit a midpoint, then $\ms$ is a geodesic space.

\end{itemize}
\end{pro}

\subsection{Model spaces}

Given a real number $\kappa\in\R$, a geodesic space of special interest is the (complete and simply connected) $2$-dimensional Riemannian manifold with constant sectional curvature $\kappa$.
For given $\kappa\in\R$, this metric space $(M^2_{\kappa},d_{\kappa})$ is unique up to an isometry, and modelled as follows.\begin{itemize}
\item If $\kappa<0$, $(M^2_{\kappa},d_{\kappa})$ is the hyperbolic plane with metric multiplied by $1/\sqrt{-\kappa}$.
\item If $\kappa=0$, $(M^2_{0},d_{0})$ is the Euclidean plane equipped with its Euclidean metric.
\item If $\kappa>0$, $(M^2_{\kappa},d_{\kappa})$ is the Euclidean sphere in $\R^3$ of radius $1/\sqrt{\kappa}$ with the angular metric.
\end{itemize}
The diameter $\varpi_{\kappa}$ of $M^{2}_{\kappa}$ is 
\[\varpi_{\kappa}:=\left\{\begin{array}{cc}
     +\infty&\mbox{if}\quad\kappa\le0,\\
     \pi/\sqrt{\kappa}&\mbox{if}\quad\kappa>0.
\end{array}\right.\]
For $\kappa\in \R$, there is a unique geodesic connecting $x$ to $y$ in $(M^2_{\kappa},d_{\kappa})$ provided $d_{\kappa}(x,y)<\varpi_{\kappa}$.
By convention, we call triangle in $M^2_{\kappa}$ any set of three distinct points $\{p,x,y\}\subset M^2_{\kappa}$, with perimeter
\[
\textrm{peri}\{p,x,y\}:=d_{\kappa}(p,x)+d_{\kappa}(p,y)+d_{\kappa}(x,y)<2\varpi_{\kappa}.
\] 
Side lengths of triangle $\{p,x,y\}$ are the numbers $d_{\kappa}(p,x)$, $d_{\kappa}(p,y)$ and $d_{\kappa}(x,y)$.
Given $a,b,c>0$ satisfying the triangle inequality and such that $a+b+c<2\varpi_{\kappa}$, there exists a unique (up to an isometry) triangle $\{p,x,y\}$ in $M^2_{\kappa}$ such that $d_{\kappa}(p,x)=a$, $d_{\kappa}(p,y)=b$ and $d_{\kappa}(x,y)=c$.
The angle $\sphericalangle^{\kappa}_{p}(x,y)$ at $p$ in $\{p,x,y\}\subset M^2_{\kappa}$ is defined by
\[\cos\sphericalangle^{\kappa}_{p}(x,y):=\left\{
 \begin{array}{ll}
 \dfrac{a^2+b^2-c^2}{2ab}&\mbox{ if }\kappa=0,\vspace{0.2cm}\\
 \dfrac{c_{\kappa}(c)-c_{\kappa}(a)\cdot c_{\kappa}(b)}{\kappa\cdot s_{\kappa}(a)s_{\kappa}(b)}&\mbox{ if }\kappa\ne 0,
 \end{array}
 \right.\]
  where $a=d_{\kappa}(p,x)$, $b=d_{\kappa}(p,y)$, $c=d_{\kappa}(x,y)$ and $c_{\kappa}:=s'_{\kappa}$ with
 \begin{equation}
 \label{def:csk}
 s_{\kappa}(r):=\left\{
 \begin{array}{ll}
 \sin(r\sqrt\kappa)/\sqrt\kappa &\mbox{ if }\kappa>0,\\
 \sinh(r\sqrt{-\kappa})/\sqrt{-\kappa}&\mbox{ if }\kappa< 0.
 \end{array}
 \right.
 \end{equation}
We end by observing that the angle is constant along geodesics in the model space $(M^2_{\kappa},d_{\kappa})$.  
\begin{pro}
\label{pro:cstangle}
Let $\kappa\in\R$ and $\{p,x,y\}\subset(M^2_{\kappa},d_{\kappa})$ be a triangle.
If $\gamma_{x}$ and $\gamma_{y}$ are geodesics from $p$ to $x$ and from $p$ to $y$ respectively, then for all $(s,t)\in(0,1]^2$,
\[\sphericalangle^{\kappa}_p(\gamma_x(s),\gamma_y(t))=\sphericalangle^{\kappa}_p(x,y).\]
\end{pro}

\subsection{Curvature}
\label{subsec:curvature}
In this section, we describe the notion of curvature bounds of metric spaces.
Curvature bounds in general metric spaces are defined by comparison arguments involving the model surfaces $(M^2_{\kappa},d_{\kappa})$ discussed in the previous section.
The fundamental device allowing for this comparison is that of a comparison triangle.
Given a metric space $(\ms,d)$, we define a triangle in $\ms$ as any set of three points $\{p,x,y\}\subset \ms$. For $\kappa\in\R$, a comparison triangle for $\{p,x,y\}$ in $M^2_{\kappa}$ is an isometric embedding of $\{p,x,y\}$ in $M^2_{\kappa}$, i.e. a set $\{p_{\kappa},x_{\kappa},y_{\kappa}\}\subset M^2_{\kappa}$ such that  
\[d_{\kappa}(p_{\kappa},x_{\kappa})=d(p,x),\quad d_{\kappa}(p_{\kappa},y_{\kappa})=d(p,y)\quad\mbox{and}\quad d_{\kappa}(x_{\kappa},y_{\kappa})=d(x,y).\] 
Such a comparison triangle always exists (and is unique up to an isometry) provided
 \[\textrm{peri}\{p,x,y\}:=d(p,x)+d(p,y)+d(x,y)< 2\varpi_{\kappa}.\] 
We are now in position to define curvature bounds for geodesic spaces.
\begin{defi}
\label{def:curvk}
Let $\kappa\in \R$ and $(\ms,d)$ be a geodesic space.
 \begin{itemize}
\item[$(1)$] We say that ${\rm curv}(\ms)\ge\kappa$ if for any triangle $\{p,x,y\}\subset \ms$ satisfying ${\rm peri}\{p,x,y\}< 2\varpi_{\kappa}$, any comparison triangle $\{p_{\kappa},x_{\kappa},y_{\kappa}\}\subset M^2_{\kappa}$, any geodesic $\gamma$ joining $x$ to $y$ in $\ms$ and any geodesic $\gamma_{\kappa}$ joining $x_{\kappa}$ to $y_{\kappa}$ in $M^2_{\kappa}$ , we have for all $t\in[0,1],$
\begin{equation}
\label{boundkappabelow}
d(p,\gamma(t))\ge d_{\kappa}( p_{\kappa},\gamma_{\kappa}(t)).
\end{equation}
\item[$(2)$] We say that ${\rm curv}(\ms)\le\kappa$ if the above definition holds with opposite inequality in \eqref{boundkappabelow}.
\end{itemize}
\end{defi} 
\noindent The previous definition has a natural geometric interpretation: if ${\rm curv}(\ms)\ge\kappa$ (resp. ${\rm curv}(\ms)\le\kappa$) a triangle $\{p,x,y\}$ looks thicker (resp. thiner) than a corresponding comparison triangle $\{p_{\kappa},x_{\kappa},y_{\kappa}\}$ in $M^2_{\kappa}$.
In the context of $\kappa=0$, the above definition may be given an alternative form of practical interest.

\begin{pro}
\label{pro:NNC}
Let $(\ms,d)$ be a geodesic space.
Then ${\rm curv}(\ms)\ge0$ if, and only if, for any points $p,x,y\in \ms$ and any geodesic $\gamma$ joining $x$ to $y$, we have
\[\forall t\in[0,1],\quad d(p,\gamma(t))^2\ge (1-t)d(p,x)^2+td(p,y)^2-t(1-t)d(x,y)^2.\]
We have ${\rm curv}(\ms)\le0$ if, and only if, the same statement holds with opposite inequality.
\end{pro}
\noindent The proof follows immediately from Definition \ref{def:curvk} by exploiting the geometry of the Euclidean plane.
Note indeed that, whenever $\{p,x,y\}\subset\R^2$ and $\R^2$ is equipped with the Euclidean metric $\|.-.\|$, the unique geodesic from $x$ to $y$ is $\gamma(t)=(1-t)x+ty$ and, for all $t\in [0,1]$, 
\[
\|p-\gamma(t)\|^2= (1-t)\| p-x\|^2+t\| p-y\|^2-t(1-t)\| x- y\|^2.
\]
For $\kappa\ne 0$, an equivalent formulation of Definition \ref{def:curvk}, given only in terms of the ambient metric $d$, is given in the next subsection using the notion of angle.
A (complete) geodesic space $(\ms,d)$ with ${\rm curv}(\ms)\le\kappa$ for some $\kappa\ge 0$ is sometimes called a CAT($\kappa$) space in reference to contributions of E. Cartan, A.D. Alexandrov and V.A. Toponogov.
A CAT($0$) space is also referred to as an NPC (non positively curved) space or an Hadamard space. Similarly, $\ms$ is also called an PC (positively curved) space if ${\rm curv}(\ms)\ge0$. If $(\ms,d)$ is a Riemannian manifold (complete for instance) with sectional curvature lower (resp. upper) bounded by $\kappa$ at every point, then $\mathrm{curv}(\ms)\ge \kappa$ (resp $\le \kappa$) in the sense of Definition \ref{def:curvk}. It is worth noting that the previous definitions are of global nature as they require comparison inequalities to be valid for all triangles (that admit a comparison triangle in the relevant model space).
Some definitions of curvature require the previous comparison inequalities to hold only locally.
The local validity of these comparison inequalities is known, under suitable conditions depending on the value of $\kappa$, to imply their global validity.
Results in this direction are known as globalisation theorems.

\subsection{Angles and space of directions}
\label{subsec:angles}
Angles, as defined below, allow to provide alternative characterisations of curvature bounds.
Let $(\ms,d)$ be a metric space and let $\kappa\in \R$.
Given a triangle $\{p,x,y\}$ in $M$ with $\textrm{peri}\{p,x,y\}<2\varpi_{\kappa}$, we define the comparison angle $\sphericalangle^{\kappa}_p(x,y)\in[0,\pi]$ at $p$ by 
 \[\cos\sphericalangle^{\kappa}_p(x,y):=\left\{
 \begin{array}{ll}
 \dfrac{d(p,x)^2+d(p,y)^2-d(x,y)^2}{2d(p,x)d(p,y)}&\mbox{ if }\kappa=0,\vspace{0.2cm}\\
 \dfrac{c_{\kappa}(d(x,y))-c_{\kappa}(d(p,x))\cdot c_{\kappa}(d(p,y))}{\kappa\cdot s_{\kappa}(d(p,x))s_{\kappa}(d(p,y))}&\mbox{ if }\kappa\ne 0,
 \end{array}
 \right.\]
  where $c_{\kappa}$ and $s_{\kappa}$ are as in \eqref{def:csk}.
In other words, given any comparison triangle $\{p_{\kappa},x_{\kappa},y_{\kappa}\}$ of $\{p,x,y\}$ in $ M^2_{\kappa}$, 
\[\sphericalangle^{\kappa}_p(x,y)=\sphericalangle^{\kappa}_{p_{\kappa}}(x_{\kappa},y_{\kappa}).\]

This allows to give an equivalent definition of curvature lower bounds that has the advantage of making sense on arbitrary metric spaces, not necessarily geodesic.
\begin{defi}[Quadruple comparison]
\label{def:quadcomp}
Let $(M,d)$ be a metric space.
Let $\kappa\in\R$.
We say that $\curv(M)\ge \kappa$, if for any four disctinct points $p,x,y,z\in M$ such that every three points have a perimeter less than $\varpi_\kappa$,
\[
\sphericalangle^{\kappa}_p(x,y)+\sphericalangle^{\kappa}_p(y,z)+\sphericalangle^{\kappa}_p(z,x)\le 2\pi.
\]
\end{defi}
\begin{pro}\label{pro:nongeoboundedcurv}
If $(M,d)$ is a geodesic space, then curvature lower bounds as defined in  Definition \ref{def:curvk} and Definition \ref{def:quadcomp} are equivalent.
Moreover, even if $(M,d)$ is not geodesic, and satisfies Definition \ref{def:quadcomp}, then, equation \eqref{boundkappabelow} is satisfied for all $t\in[0,1]$ and $\gamma(t)\in \ms$ such that 
\[
d(\gamma(0),\gamma(t))/t=d(\gamma(t),\gamma(1))/(1-t)=d(\gamma(1),\gamma(0)).
\]
\end{pro}

The next result presents a characterization of curvature bounds in terms of the monotonicity of the comparison angle.
\begin{pro}[Angle monotonicity]\label{pro:monotone}
Let $(\ms,d)$ be a geodesic space and let $\kappa\in \R$.
Then ${\rm curv}(\ms)\ge \kappa$ (resp. ${\rm curv}(\ms)\le \kappa$), in the sense of Definition \ref{def:curvk}, if and only if, for any triangle $\{p,x,y\}$ in $M$ and any geodesics $\gamma_{x}$ and $\gamma_y$ from $p$ to $x$ and from $p$ to $y$ respectively, the function
 \[(s,t)\in[0,1]^2\mapsto\sphericalangle^{\kappa}_p(\gamma_x(s),\gamma_y(t)),\]
 is non-increasing (resp. non-decreasing) in each variable when the other is fixed.
 \end{pro}
 
 In a geodesic space $(\ms,d)$, if $p\in \ms$ and $\gamma_x$ and $\gamma_y$ are two geodesics connecting $p$ to $x$ and $y$ respectively, we define
 \[
 \sphericalangle^{\kappa}(\gamma_x,\gamma_y):=\lim_{s,t\to 0}\sphericalangle^{\kappa}_p(\gamma_x(s),\gamma_y(t)),
 \]
 when it exists.
 It follows from Proposition \ref{pro:monotone} that this limit exists provided ${\rm curv}(\ms)\ge \kappa$ or ${\rm curv}(\ms)\le \kappa$.
It may be shown furthermore that this limit is independent of $\kappa$.
Hence, whenever $\ms$ has upper or lower curvature bound, we denote $\sphericalangle(\gamma_x,\gamma_y)$ the angle between these two geodesics. Given a third geodesic $\gamma_z:[0,1]\to S$ such that $\gamma_z(0)=p$ and $\gamma_z(1)=z$, we have the triangular inequality \begin{equation}
 \label{angleti}
\sphericalangle(\gamma_x,\gamma_y)\le \sphericalangle(\gamma_x,\gamma_z)+\sphericalangle(\gamma_z,\gamma_y),
 \end{equation}
 so that $\sphericalangle$ defines a pseudo metric on the set $G(p)$ of all geodesics emanating from $p$.
Defining the equivalence relation $\sim$ on $G(p)$ by $\alpha\sim\beta\Leftrightarrow \sphericalangle(\alpha,\beta)=0$, the angle $\sphericalangle$ induces a metric (still denoted $\sphericalangle$) on the quotient set $G(p)/\sim$ and we call space of directions the completion $(\Sigma_p,\sphericalangle)$ of $(G(p)/\sim,\sphericalangle)$.
An element of $\Sigma_p$ is called a direction.

\subsection{Tangent cones}
\label{subsec:cone}
Metric spaces considered so far have a priori no differentiable structure.
In this context, an analog of a tangent space is provided by the notion of a tangent cone.
This section shortly reviews this notion.
Below, $M$ denotes a geodesic space with lower or upper bounded curvature in the sense of Definition \ref{def:curvk}.
  
\begin{defi}[Tangent cone]
Let $p\in \ms$.
The tangent cone $T_p\ms$ at $p$ is the Euclidean cone over the space of directions $(\Sigma_p,\sphericalangle)$.
In other words, $T_p\ms$ is the metric space:\begin{itemize} 
\item Whose underlying set consists in equivalent classes in $\Sigma_p\times[0,+\infty)$ for the equivalence relation $\sim$ defined by 
\[
(\alpha,s)\sim (\beta,t) \Leftrightarrow ((s=0\mbox{ and }t=0)\mbox{ or }(s=t\mbox{ and }\alpha=\beta)).
\]
A point in $T_p\ms$ is either the tip of the cone $o_p$, i.e. the class $\Sigma_p\times\{0\}$, or a couple $(\alpha,s)\in\Sigma_p\times(0,+\infty)$ (identified to the class $\{(\alpha,s)\}$).
\item Whose metric $d_p$ is defined (without ambiguity) by
\[
d_p((\alpha,s),(\beta,t)):=\sqrt{s^2+t^2-2st\cos\sphericalangle(\alpha,\beta)}.
\]
\end{itemize}
\end{defi}
For $u=(\alpha,s)$ and $v=(\beta,t)\in T_p\ms$, we often denote $\|u-v\|_p:=d_p(u,v)$, $\|u\|_p:=d_p(o_p,u)=s$ and \[ \langle u,v\rangle_p:=st\cos\sphericalangle(\alpha,\beta)=(\|u\|^2_p+\|v\|^2_p-\|u-v\|^2_p)/2.\]

\begin{pro}
 Let $(\ms,d)$ be a geodesic space and $p\in \ms$ be fixed.
If ${\rm curv}(\ms)\ge \kappa$ for some $\kappa\in \R$, then $T_p\ms$ is a metric space with ${\rm curv}(T_p\ms)\ge 0$ (in the sense of Definition \ref{def:quadcomp}).
\end{pro}

Note that the tangent space is not always a geodesic space, see discussion before Proposition 28 in \cite{yokota2012rigidity} and the Proposition itself for the proof.
Notation $\|.\|_p$ and $\langle.,.\rangle_p$ introduced above is justified by the fact that the cone $T_p\ms$ possesses a Hilbert-like structure described as follows.
For a point $u=(\alpha,t)$ and $\lambda\ge 0$, we define $\lambda\cdot u:=(\alpha, \lambda t)$.
Then, the sum of points $u,v\in T_p\ms$ is defined as the mid-point of $2\cdot u$ and $2\cdot v$ as defined in Definition \ref{def:midpoint}.
Finally, it may be checked using the previous definitions that, for any $u,v\in T_p\ms$ and any $\lambda\ge 0$, we get 
\[\|\lambda\cdot u\|_p=\lambda\|u\|_p\quad\mbox{and}\quad \langle \lambda\cdot u,v\rangle_p=\langle u,\lambda\cdot v\rangle_p=\lambda\langle u,v\rangle_p.\]

Next we define logarithmic maps.
Fix $p\in\ms$.
Since $\ms$ has upper or lower bounded curvature, the angle monotonicity imposes the following observation.
If there are two geodesics $\gamma^1_x$ and $\gamma^2_x$ connecting $p$ to $x$ such that $\sphericalangle(\gamma^1_x,\gamma^2_x)=0$, then $\gamma^1_{x}=\gamma^2_{x}$. In other words, the set of points $x\in\ms$ for which there is not only one equivalence class of geodesics connecting $p$ to $x$ is exactly the set of points $x$ connected to $p$ by at least two distinct geodesics. This set of points is denoted $C(p)$ and called the cut-locus of $p$. For all $x\in\ms\setminus C(p)$, we denote $\uparrow^x_{p}$ the direction of the unique geodesic $\gamma_x:[0,1]\to\ms$ connecting $p$ to $x$.

\begin{defi}[Logarithmic map]
Let $(\ms,d)$ be a geodesic space with curvature bounded from above or below in the sense of Definition \ref{def:curvk} and fix $p\in \ms$.
We denote
 \[
 \log_p: x\in\ms\setminus C(p)\mapsto (\uparrow^{x}_p,d(p,x))\in T_p\ms.
 \]
  \end{defi}
  
For all $t\in[0,1]$ and all $x\in \ms\setminus C(p)$, one checks in particular that
 \[\log_p(\gamma_{x}(t))=t\cdot \log_p(x).\]
 More generally, if $\gamma:[0,1]\to\ms$ is a geodesic in $\ms$ and if we denote $p=\gamma(t)$ for some $t\in[0,1]$, then provided $\gamma(0)=x$ and $\gamma(1)=y$ both belong to $\ms\setminus C(p)$, we check that, for all $s\in[0,1]$,
 \[\log_p(\gamma(s))=(1-s)\cdot\log_p(x)+s\cdot\log_p(y).\]
Note finally that $C(p)$ is empty for spaces of curvature bounded from above by $0$. It is often practical to extend the definition of $\log_p$ to the cut locus $C(p)$. To that aim, it suffices to select, for any $x\in C(p)$, a geodesic connecting $p$ to $x$. 
Denoting as above $\uparrow_p^x$ the direction associated to this chosen geodesic, we simply extend $\log_p$ to $M$ by setting $\log_p(x)=(\uparrow^{x}_p,d(p,x))$. This extension depends a priori on the choice of geodesics connecting $p$ to $x\in C(p)$. However, note that when $(M,d)$ is a Polish space, this extension $\log_p$ can be chosen to be measurable on $M$, by application of the measurable selection theorem of Kuratowski and Ryll-Nardzewski (see Theorem 6.9.3 in \cite{Bog2007}). Next is a fundamental result.

  \begin{pro}\label{lem:tangentcone2}
  Let $(\ms,d)$ be a geodesic space and $p\in \ms$ be fixed.
If ${\rm curv}(\ms)\ge0,$
then for all $x,y\in \ms$,
 \[d(x,y)\le\|\log_p(x)-\log_p(y)\|_p,\]
 with equality if $x=p$ or $y=p$.
 \end{pro}

\section*{Acknowledgements}
We would like to thank the associate editor, as well as two anonymous referees, for valuable comments that helped improve significantly the original version of the manuscript. We also express our gratitude to Philippe Rigollet for stimulating conversations and useful remarks on a preliminary version of the paper.

\bibliographystyle{alphaabbr}
\bibliography{biblio}

\end{document}